\documentclass{amsart}
\usepackage[utf8]{inputenc}
\usepackage[english]{babel}
\usepackage{graphicx}
\usepackage{array}
\usepackage{amsmath}
\usepackage{amsthm}
\usepackage{amsfonts} 
\usepackage[T1]{fontenc}
\usepackage{graphicx}
\usepackage[bottom=3cm, left=2.5cm, right=2.5cm, top=2.5cm]{geometry}
\usepackage{color} 
\usepackage{cite} 

\newtheorem{lemat}{Lemma}[section]
\newtheorem{tw}{Theorem}[section]
\newtheorem{ass}{Assumption}[section]
\newtheorem{defi}{Definition}[section]
\newtheorem*{uwaga}{Remark}
\numberwithin{equation}{section}

\def\vc#1{\boldsymbol{#1}}
\def\ten#1{\boldsymbol{#1}}

\newcommand{\dxdt}{\,{\rm{d}}x\,{\rm{d}}t}
\newcommand{\dx}{\,{\rm{d}}x}

\newcommand{\ds}{\,{\rm{d}}s}

\newcommand{\dt}{\,{\rm{d}}t}
\newcommand{\dtau}{\,{\rm{d}}\tau}
\newcommand{\braket}[1]{(\!(#1)\!)}

\def\div{\rm{div\,}}

\title{Thermo-visco-elasticity for Norton-Hoff-type models with homogeneous thermal expansion}

\author[P. Gwiazda]{Piotr Gwiazda}
\address{
Institute of Mathematics, Polish Academy of Sciences, \'{S}niadeckich 8, 00-656 Warsaw, Poland
\newline
Institute of Applied Mathematics and Mechanics, University of Warsaw,  Banacha 2, 02-097 Warsaw, Poland}
\email{pgwiazda@mimuw.edu.pl}

\author[F. Z. Klawe]{Filip Z. Klawe}
\address{Institute of Applied Mathematics, Heidelberg University, Im Neuenheimer Feld 205, 69120 Heidelberg, Germany
\newline
Institute of Applied Mathematics and Mechanics, University of Warsaw,  Banacha 2, 02-097 Warsaw, Poland}
\email{fzklawe@mimuw.edu.pl}

\author[S. Owczarek]{Sebastian Owczarek}
\address{Faculty of Mathematics and Information Science, Warsaw University of Technology, Koszykowa 75, 00-662 Warsaw, Poland}
\email{s.owczarek@mini.pw.edu.pl}

\thanks{Acknowledgement: 
P.G. received support from the National Science Centre (Poland), 2015/18/M/ST1/00075.
F.Z.K. is thankful to MAThematics Center Heidelberg (MATCH). F.Z.K. was partially supported by the NCN grant, the number of decisions 2014/13/N/ST1/02626.
S.O. is thankful to the Warsaw Center of Mathematics and Computer Science (WCMCS) for partially support during his stay on the postdoctoral research position at the Institute of Applied Mathematics and Mechanics, University of Warsaw.  }

\begin{document}

\begin{abstract}
In this work we study a quasi-static evolution of thermo-visco-elastic model with homogeneous thermal expansion. We assume that material is subject to two kinds of  mechanical deformations: elastic and inelastic. Inelastic deformation is related to a hardening rule of Norton-Hoff type. Appearance of inelastic deformation causes transformation of mechanical energy into thermal one, hence we also take into the consideration changes of material's temperature. 

The novelty of this paper is to take into account the thermal expansion of material. We are proposing linearisation of the model for homogeneous thermal expansion, which preserves symmetry of system and therefore total energy is conserved. Linearisation of material's thermal expansion is performed in definition of Cauchy stress tensor and in heat equation. In previous studies, it was done in different way. Considering of such linearisation leads to system where the coupling between temperature and displacement occurs in two places, i.e. in the constitutive function for the evolution of visco-elastic strain and in the additional term in the heat equation, in comparison to models without thermal expansion. The second coupling was not considered previously. For such system of equations we prove the existence of solutions. Moreover, we obtain existence of displacement's time derivative, which has not been done previously.

\end{abstract}

\maketitle

\section{Introduction}

The subject of this work is to analyze the class of models describing response of thermo-visco-elastic material to applied external forces and the heat flux through the boundary. Thermo-visco-elastic system of equations captures displacement, temperature and visco-elastic strain of the body. It is a consequence of physical principles, such as balance of momentum and balance of energy, cf. \cite{GreenNaghdi,1,LandauLifshitz}, supplemented by two constitutive relations: definition of Cauchy stress tensor and evolutionary equation for visco-elastic strain, which describes the material properties.

Reactions of visco-elastic materials may be different for different loads speed. Our interest is to examine slow motion of materials where inertial forces are negligible, see e.g. \cite{ChelRacke,DuvautLions,johnson1,suquet1,suquet2,temam1,temam2}. Additionally, we consider the model with infinitesimal displacement (dependence between the Cauchy stress tensor and elastic part of strain is linear, i.e. generalized Hooke’s law holds), the process holds in the neighborhood of some reference temperature and, what is new here, we consider problem which also takes into account thermal expansion of material.

We  assume that the  body $\Omega \subset \mathbb{R}^3$ is an open bounded set with a $C^2$ boundary and moreover, it is homogeneous in space. The material undergoes two kinds of  deformations: elastic and inelastic. By the first type we understand reversible deformations, by the second - irreversible. In this paper we deal with visco-elastic type of inelastic deformation.
The problem is captured by the following system
\begin{equation}
\begin{aligned}
- \div \ten{\sigma} &= \vc{f} ,
\\
\ten{\sigma} &= \ten{T} -\alpha \ten{I},
\\
\ten{T} &= \ten{D}(\ten{\varepsilon}(\vc{u}) - \ten{\varepsilon}^{\bf p} ),
\\
\ten{\varepsilon}^{\bf p}_t &= \ten{G}(\theta,\ten{T}^d) ,
\\
\theta_t - \Delta\theta + \alpha {\div} \vc{u}_t&= \ten{T}^d:\ten{G}(\theta,\ten{T}^d) ,
\end{aligned}
\label{full_system}
\end{equation}
which are fulfilled in $\Omega\times(0,T)$ and where $\vc{u}:\Omega\times\mathbb{R}_+\rightarrow \mathbb{R}^3$ describes displacement of the material, $\theta:\Omega\times\mathbb{R}_+\rightarrow\mathbb{R}$ stays for temperature of the material and $\ten{\varepsilon}^{\bf p}:\Omega\times\mathbb{R}_+\rightarrow \mathcal{S}^3_d$ is visco-elastic strain tensor. We denote by $\mathcal{S}^3$ the set of symmetric $3 \times 3$-matrices with real entries and by $\mathcal{S}^3_d$ a subset of $\mathcal{S}^3$ which contains traceless matrices. By $\ten{T}^d$ we mean the deviatoric part (traceless) of the tensor $\ten{T}$, i.e. $\ten{T}^d=\ten{T}-\frac{1}{3}tr(\ten{T})\ten{I}$, where $\ten{I}$ is the identity matrix from $\mathcal{S}^3$. Additionally, $\ten{\varepsilon}(\vc{u})$ denotes the symmetric part of the gradient of displacement~$\vc{u}$, i.e. $\ten{\varepsilon}(\vc{u})=\frac{1}{2}(\nabla\vc{u} + \nabla^T\vc{u})$.  The volume force is denoted by $\vc{f}:\Omega\times\mathbb{R}_+\rightarrow \mathbb{R}^3$. 

We complete the considered problem by formulating the initial conditions
\begin{equation}
\begin{aligned}
\theta(x,0)&=\theta_0(x), 
\\
\ten{\varepsilon}^{\bf p}(x,0)&=\ten{\varepsilon}^{\bf p}_0(x),
\end{aligned}
\label{init_0}
\end{equation}
in $\Omega$ and boundary conditions
\begin{equation}
\begin{aligned}
\vc{u}&=\vc{g}, \\
\frac{\partial \theta}{\partial \vc{n}}&=g_{\theta},
\end{aligned}
\label{boun_0}
\end{equation}
on $\partial\Omega\times(0,T)$.

The visco-elastic strain tensor is described by the evolutionary equation with prescribed constitutive function $\ten{G}(\cdot,\cdot)$. Different assumptions made on function $\ten{G}(\cdot,\cdot)$ lead to creation of different models. The subject of current paper is to consider the hardening rule defined by Norton-Hoff-type constitutive law, see forthcoming Assumption \ref{ass_G}. In the literature, many different models were described, see \cite{DuvautLions,Alber,ChelRacke,ChelNeffOwczarek14,owcz2} or \cite{PhDFilip,1}. Norton-Hoff or Norton-Hoff-type models were studied e.g. in \cite{1,KlaweOwczarek, ChelminskiOwczarekthermoI, ChelminskiOwczarekthermoII}

The function $\ten{\sigma}:\Omega\times\mathbb{R}_+\rightarrow \mathcal{S}^3$ is the Cauchy stress tensor. It may be divided into two parts: mechanical (elastic) and thermal one. The mechanical part is $\ten{T}=\ten{D}(\ten{\varepsilon}(\vc{u}) - \ten{\varepsilon}^{\bf p})$, where the operator $\ten{D}:\mathcal{S}^3\rightarrow\mathcal{S}^3$ is linear, positively definite and bounded. 
The operator $\ten{D}$ is a four-index matrix, i.e. $\ten{D}=\left\{d_{i,j,k,l}\right\}_{i,j,k,l=1}^3$ and the following equalities  hold
\begin{equation}
d_{i,j,k,l} = d_{j,i,k,l},
\quad
d_{i,j,k,l} = d_{i,j,l,k}
\quad
\mbox{and}
\quad
d_{i,j,k,l} = d_{k,l,i,j}
\quad
\forall i,j,k,l=1,2,3 .
\end{equation}
These equivalences are consequences of angular momentum conservation. For simplicity, we assume that $d_{i,j,k,l}$ are constants. 

The second part of Cauchy stress tensor is a thermal one. This term did not appear in \cite{1,tve-Orlicz,GKSG}. The novelty of this paper is that we added thermal expansion to model considered previously, which extends result for more general class of models. Taking thermal expansion into account, heat equation also changes and it may be understood as follows: an additional transfer of energy between its mechanical and thermal parts appears as a consequence of material's thermal expansion. There appears a nonlinear term $\alpha(\theta-\theta_R) \div \vc{u}_t$ in heat equation (derivation of thermo-visco-elastic model may by found e.g. in \cite{PhDFilip}), which causes main problems during the analysis. Here, we denote by $\theta_R$ a temperature in which thermal stress is equal to zero. There were many papers in which such issue was presented. Authors deal with it in different ways, e.g. by linearisation of nonlinear term, see \cite{JR,Haupt,Bartczak}, or by adding additional damping term in momentum equation, see \cite{ChelminskiOwczarekthermoI,ChelminskiOwczarekthermoII}, which  gives missing regularity estimates for time derivative of displacement as a consequence of estimate for momentum equation. Without this additional damping term time derivative of displacement appears only in heat equation. However, none of these ways guarantee that model catch all of mathematical and physical properties of considered phenomenon. Our idea is to make different linearisation than one presented e.g. \cite{JR,Haupt,Bartczak}.

Now, we shall explain our reasons for considering thermal part of Cauchy stress tensor in form of $\alpha\ten{I}$ and additional term in heat equation in form of $\alpha\div \vc{u}_t$. Model derivation, see \cite{PhDFilip,LandauLifshitz}, leads to obtaining strictly different form of \eqref{full_system}. However, making suitable assumptions (slow and long-time behaviour of materials, process holds in the neighborhood of reference temperature etc.) it may be considered in this form. The first assumption leads to omitting the acceleration term in momentum equation, whereas the second one gives us opportunity to make a thermal part of Cauchy stress tensor simplified.

There is a big class of materials which are not subject of thermal expansion, that is $\alpha=0$. This case was a subject of our  previous studies, see \cite{1,tve-Orlicz,GKSG}. There are also materials which change their volume with changes of temperature. If their volume increases with increasing temperature then $\alpha>0$, otherwise $\alpha<0$. Moreover, thermal expansion of material depends on conditions in which material is examined, e.g. it may depend on temperature, pressure etc..

At this moment we should distinguish between two values of temperatures, which will be subject of the following discussion. The first one is temperature for which thermal stress does not appear. We denote it by $\theta_R$. The second one is temperature in the neighborhood of which the process holds. We call it a reference temperature and denote it by $\bar{\theta}$. We assume that thermal strain is proportional to difference between temperature $\theta$ and $\theta_R$, i.e. it is equal to $\alpha(\theta -\theta_R)\ten{I}$, where $\alpha$ is constant (positive or negative). Materials with such properties were subject of study in \cite{Lyon,wachtman,Martinek} and many others. Of course, thermal stress may be defined more generally, where $\alpha(\cdot)$ is a smooth function of $\theta-\theta_R$, such that $\alpha(0)=0$. The following reasoning will work also in this case. However, we will focus on linear dependency. Then, assumption on form of thermal stress causes that there appears also coupling in heat equation and we have
\begin{equation}
\begin{aligned}
\ten{\sigma} & = \ten{T} - \alpha (\theta -\theta_R)\ten{I},
\\
\theta_t - \kappa\Delta\theta +\alpha (\theta -\theta_R) \div \vc{u}_t  & = \ten{T}^d:\ten{G}(\theta,\ten{T}^d).
\end{aligned}
\label{eq:1.5}
\end{equation}
The main issue which we have to deal with here is nonlinear term in heat equation, i.e. $\alpha\theta \div \vc{u}_t$. Linearisation of this term solves this problem. Appearance of this term in heat equation is a consequence of definition of Cauchy stress tensor. Hence, if we assume that process holds in the neighborhood of temperature $\bar{\theta}$ and we linearise temperature in term $\alpha(\theta-\theta_R) \div \vc{u}_t$ in heat equation without making a linearisation of Cauchy stress tensor, we will lose the symmetry in system of equations. Similar assumptions were done in \cite{JR,Haupt}, where authors assume that $(\theta -\theta_R)- \bar{\theta}$ is sufficiently small only in heat equation. Then it leads to
\begin{equation}
\begin{aligned}
\ten{\sigma} &= \ten{T}  -\alpha (\theta - \theta_R) \ten{I} ,
\\
\theta_t - \kappa\Delta\theta + \gamma \div \vc{u}_t&= \ten{T}^d:\ten{G}(\theta,\ten{T}^d),
\end{aligned}
\end{equation}
where $\gamma$ is a constant which approximate $\alpha(\theta - \theta_R)$. This linearisation leads to the system where energy is not conserved because  of broken symmetry between definition of Cauchy stress tensor and heat equation. Our idea is to make linearisation in the different way, to avoid this unexpected property of system which describe physical phenomenon. We linearise term $\theta-\theta_R$ in both equations \eqref{eq:1.5}, i.e. we assume that $\theta-\theta_R$ may be approximated by $\bar{\theta}$. Thus,
\begin{equation}
\begin{aligned}
\ten{\sigma} &= \ten{T}  -\alpha \bar{\theta} \ten{I}, 
\\
\theta_t - \kappa\Delta\theta + \alpha\bar{\theta} \div \vc{u}_t&= \ten{T}^d:\ten{G}(\theta,\ten{T}^d) .
\end{aligned}
\label{eq:1.9}
\end{equation}
Henceforth, we focus on equation for Cauchy stress tensor and heat equation in this form. It has not been examined previously. One should notice two aspects of such linearisation. Firstly, instead of \cite{Haupt,JR} it leads to the system which conserves the energy. And secondly, in our linearisation there appears non-zero stress for homogeneous data problem. It may be understood as constant pressure caused by external force, i.e. the material is compressed (for positive $\alpha\bar{\theta}$) or stretched (for negative $\alpha\bar{\theta}$). Furthermore, approximation (Taylor's series) made for temperature is cut off on the same level in all equations which do not take place in \cite{Haupt,JR,Bartczak}. Since $\alpha,\bar{\theta}$ are constants, in the rest of the paper, we will denote the product of $\alpha\bar{\theta}$ by $\alpha$. Additionally, we assume that it is positive.

It is worth to mention \cite{ChelminskiOwczarekthermoI} and \cite{ChelminskiOwczarekthermoII}, where authors considered the thermo-visco-elastic system of equations with non-linear thermal expansion. In those papers authors did not assume that process holds in the neighborhood of reference temperature as occurred in the system \eqref{eq:1.9}. We use the original notation from those paper but forthcoming function $f$ is the same as considered here function $\alpha$. Thermal part of Cauchy stress tensor is equal to $-f(\theta)\ten{I}$, where $f:\mathbb{R}\rightarrow\mathbb{R}$ is continuous and satisfies suitable growth conditions (motivation for assumptions on the function $f$ is similar to the conditions studied in \cite{BlanchardGuibe97} and \cite{BlanchardGuibe00}).  However, this nonlinear thermal part of stress imposed to add a damping term to momentum equation, which allows to control the divergence of velocity in the heat equation. In the current study, we do not add this term in system \eqref{full_system}. Moreover, we make more general assumptions on function $\ten{G}$ than in \cite{ChelminskiOwczarekthermoI,ChelminskiOwczarekthermoII}, where it did not depend on temperature and had more detailed growth conditions with respect to second variable.

Similar way to deal with such problem was presented in \cite{BR1,BR,RouBart2}, where authors studied the thermal-visco-plasticity system for Kelvin-Voigt-type material. Kelvin-Voigt-type materials have got additional term in Cauchy stress tensor, i.e. time derivative of deformations gradient, which regularise the solution. Mathematically, authors obtain a PDEs with different order and this additional term allows to control the divergence of velocity in the heat equation. Nevertheless, physical motivations for this additional term are different than ones presented in \cite{ChelminskiOwczarekthermoI,ChelminskiOwczarekthermoII} but regularisations effects are the same. In \cite{BR1,BR,RouBart2} evolution of the plastic strain is governed by Prandtl-Reuss flow rule. Material's thermal expansion appears in \cite{BR,RouBart2}. Since flow rules in these papers did not depend on temperature, the only coupling effect between displacement and temperature was caused by thermal part of Cauchy stress tensor. On the contrary, in \cite{BR1} Kelvin-Voigt material without thermal expansion was considered and a coupling between displacement and temperature was a consequence of temperatures dependent flow rule. It is worth to emphasize that in \eqref{full_system} coupling between thermal and mechanical effects takes place similarly in thermal expansion and flow rule.

\begin{ass}
The function $\ten{G}(\theta,\ten{T}^d)$ is  continuous with respect to   $\theta$ and $\ten{T}^d$ and satisfies for $p\ge 2$
the following conditions:
\begin{itemize}
\item[a)] $(\ten{G}(\theta,\ten{T}^d_1)-\ten{G}(\theta,\ten{T}^d_2)):(\ten{T}^d_1-\ten{T}^d_2) \geq 0$, for all $\ten{T}_1^d,\ten{T}_2^d \in \mathcal{S}^3_d$ and $\theta\in \mathbb{R}$;
\item[b)] $|\ten{G}(\theta,\ten{T}^d)| \leq C(1 + |\ten{T}^d|)^{p-1}$, where $\ten{T}^d\in\mathcal{S}^3_d$, $\theta\in \mathbb{R}$;
\item[c)] $\ten{G}(\theta,\ten{T}^d):\ten{T}^d \geq \beta |\ten{T}^d|^p$, where $\ten{T}^d\in\mathcal{S}^3_d$, $\theta\in \mathbb{R}$,
\end{itemize}
\label{ass_G}
where  $C$ and $\beta$ are positive constants,  independent of the temperature $\theta$.
\end{ass}
The subject of the present study is to focus on the main issues which appear during the analysis of models including thermal expansion. Norton-Hoff-type  model is a good prototype to develop general theory. It is also a good approximation of Prandtl-Reuss law of elastic-perfectly-plastic deformation, see \cite{ChelRacke,temam2}. 

One may observe that in the system \eqref{full_system} displacement and temperature depend on each other. This leads to many technical problems which we have to deal with during the analysis of this model. It would seem that the omission of explicit dependent of temperature in definition of Cauchy stress tensor leads to displacement which is independent of temperature. However, we shall observe that temperature appears in the evolutionary equation for the visco-elastic strain tensor and it have implicit impact on displacement.

Before we formulate definition of weak solutions and state the main theorem of this paper let us introduce notation $W^{1,p'}_{\vc{g}}(\Omega,\mathbb{R}^3):=\left\{\vc{u} \in W^{1,p'}(\Omega,\mathbb{R}^3): \vc{u}=\vc{g} \mbox{ on } \partial\Omega \right\}$ and $W^{1,p'}_{\vc{g_t}}(\Omega,\mathbb{R}^3):=\left\{\vc{v} \in W^{1,p'}(\Omega,\mathbb{R}^3):\vc{v}=\vc{g_t} \mbox{ on } \partial\Omega \right\}$, where $\vc{g_t}$ denotes time derivative of function $g$ and $p'=p/(p-1)$. This allows us to define the solution to thermo-visco-elastic model in transparent way.

\begin{defi}
Let $p\geq 2$ and $q\in(1,\frac{5}{4})$. 
The triple of functions 
\begin{equation}
\begin{split}
\vc{u}&\in L^{p'}(0,T,W^{1,p'}_{\vc{g}}(\Omega,\mathbb{R}^3)) \ \mbox{with}\  \vc{u}_t\in L^{p'}(0,T,W^{1,p'}_{\vc{g_t}}(\Omega,\mathbb{R}^3)),
\\
\ten{T}&\in L^2(0,T,L^2(\Omega,\mathcal{S}^3))
\end{split}
\nonumber
\end{equation}
and 
\begin{equation}
\theta\in L^q(0,T,W^{1,q}(\Omega))\cap C([0,T],W^{-2,2}(\Omega))
\nonumber
\end{equation}
is a weak solution to the system \eqref{full_system} if
\begin{equation}
\int_0^T\int_{\Omega}(\ten{T} - \alpha\ten{I}):\nabla\vc{\varphi} \dxdt 
= \int_0^T\int_{\Omega}\vc{f}\cdot \vc{\varphi} \dxdt ,
\end{equation}
where 
\begin{equation}
\ten{T}=\ten{D}(\ten{\varepsilon}(\vc{u}) - \ten{\varepsilon}^{\bf p}),
\end{equation}
and
\begin{equation}
\begin{split}
-\int_0^T\int_{\Omega} \theta\phi_t \dxdt -
\int_{\Omega} \theta_0(x)\phi(0,x) \dx 
& + \int_0^T\int_{\Omega}  \nabla\theta\cdot\nabla\phi  \dxdt -
\int_0^T\int_{\partial\Omega}g_{\theta}\phi  \dxdt 
\\
&+ \int_0^T\int_{\Omega}\alpha\div(\vc{u}_t)\phi \dxdt
 = 
\int_0^T\int_{\Omega} \ten{T}^d:\ten{G}(\theta,\ten{T}^d)\phi \dxdt,
\end{split}
\end{equation}
holds for every test function $\vc{\varphi}\in C^{\infty}([0,T],C^{\infty}_c(\Omega,\mathbb{R}^3))$ and $\phi\in C^{\infty}_c([-\infty,T),C^{\infty}(\Omega))$. Furthermore, the visco-elastic strain tensor can be recovered from the equation on its evolution, i.e.
\begin{equation}
\ten{\varepsilon}^{\bf p}(x,t) = \ten{\varepsilon}^{\bf p}_0(x) + \int_0^t \ten{G}(\theta(x,\tau),\ten{T}^d(x,\tau)) \dtau,
\end{equation}
for a.e. $x\in\Omega$ and $t\in [0,T)$. Moreover, $\ten{\varepsilon}^{\bf p} \in W^{1,p'}(0,T,L^{p'}(\Omega,\mathcal{S}^3_d))$.
\end{defi}

\begin{tw}
Let $p\geq 2$ and let initial conditions satisfy $\theta_0 \in L^1(\Omega)$, $\ten{\varepsilon}^{\bf p}_0\in L^2(\Omega,\mathcal{S}^3_d)$, boundary conditions satisfy $\vc{g}\in W^{1,p}(0,T, W^{1-\frac{1}{p},p}(\partial\Omega,\mathbb{R}^3))$, $g_{\theta}\in L^2(0,T,L^2(\partial\Omega))$ and volume force $\vc{f}\in W^{1,p}(0,T,W^{-1,p}(\Omega,\mathbb{R}^3))$ and function $\ten{G}(\cdot,\cdot)$ satisfy the  Assumption \ref{ass_G}. Then there exists a weak solution to system \eqref{full_system}.
\label{thm:main}
\end{tw}

The proof regarding existence of solutions to thermo-visco-elastic model with thermal expansions and Norton-Hoff-type hardening rule is done with use of two level Galerkin approximation. It means that we have independent parameters for approximation of displacement and temperature. This method was previously used for continuum mechanic models, e.g. see \cite{1,GKSG,tve-Orlicz,KlaweOwczarek}, or for models describing fluid motion, see \cite{BFM,BMR}. The main reason to use two level approximation here is low regularity of right-hand side of heat equation. Since product $\ten{G}(\theta, \ten{T}^d):\ten{T}^d$ is only an integrable function we have to use technique which gives us existence of solution to parabolic equation with low regular data. There are two possible approaches which may be applied here: Boccardo and Gallou\"{e}t approach or renormalised solutions, see \cite{Boccardo,BlanchardMurat,Blanchard}. We focus on Boccardo and Gallou\"{e}t, since it makes this paper more clear. However, also renormalised solutions may be applied for continuum mechanics problem with law regularity of data, see \cite{tve-Orlicz}.

We have to use two level approximation because of technical part of proof presented by Boccardo and Gallou\"{e}t. We have to test approximate heat equation by truncation of its solution. Approximate solutions are constructed as finite dimensional approximations, see Appendix \ref{sec:constr}. Construction of basis functions does not guarantee that after truncation the approximate solution will belong to the same finite dimensional space. Due to this fact we use different parameter of approximation regarding to displacement and temperature. We will firstly make a limit passage with parameter corresponding to approximation of temperature to obtain a sequence of approximate temperature in infinite dimensional space. For such functions, their truncations belong to the same space. Together with truncation of right-hand side of heat equation and initial data on level corresponding to range of Galerkin approximation for displacement, it guarantees that Boccardo and Gallou\"{e}t approach may be applied.

Since present considerations are following partially similarly as presented in \cite{1}, we skip some parts of the proof, which may be found in \cite{1}. 

All functions appearing in this paper are functions of position $x$ and time $t$. We often omit the variables of the function and write $\vc{u}$ instead of $\vc{u}(x,t)$. All of the computation are conducted in Lagrangian coordinates. In view of the fact that the displacement is small, the stress tensor in Lagrangian coordinates is approximated by the stress tensor in Eulerian coordinates. This is a standard way of considering the inelastic models, for more details see \cite[Chapter 13.2]{temammiranville}. Moreover, we denote vectors by $\vc{v}$ (small bold letters) and matrices by $\ten{T}$ (capital bold letters).

The rest of this paper is organised as follows: In Section \ref{sec:fist} we present proof of main theorem. Appendix \ref{sec:trans} is dedicated to transformation of system into homogeneous boundary value problem. In Appendix \ref{sec:constr} we present a construction of bases which are used to obtain approximate solutions. In Appendix \ref{sec:Lemma} we proved Lemma \ref{wsp_org_epa}. Finally, in Appendix \ref{sec:BG} we recalled Boccardo and Gallou\"{e}t approach to parabolic equation with Neumann boundary condition.

\section{Proof of Theorem \ref{thm:main}}
\label{sec:fist}
The aim of this section is to present proof of existence of solution regarding to Norton-Hoff-type models with thermal expansion. We focus on the main problems which appear during the analysis of models including thermal expansion. The proof is divided into a few steps. Firstly, we define the approximate solutions and prove their existence. Then, we make limit passages with $l\to \infty$ and then with $k\to \infty$.

\subsection{Construction of approximate solutions}
\label{sec:cons1}
We start the proof from transforming the problem into homogeneous boundary value problem, see Appendix \ref{sec:trans}. Hence, \eqref{full_system} transforms into
\begin{equation}
\begin{aligned}
- {\div} \ten{\sigma} & =  0 ,
\\
\ten{\sigma} & = \ten{T} -\alpha \ten{I},
\\
\ten{T} & = \ten{D}(\ten{\varepsilon}(\vc{u}) - \ten{\varepsilon}^{\bf p} ),
\\
\ten{\varepsilon}^{\bf p}_t & =  \ten{G}(\tilde{\theta} + \theta ,  \tilde{\ten{T}}^d + \ten{T}^d),
\\
\theta_t - \Delta \theta + \alpha \div (\vc{u}_t)& =   \big(\tilde{\ten{T}}^d + \ten{T}^d\big):\ten{G}(\tilde{\theta} + \theta , \tilde{\ten{T}}^d + \ten{T}^d),
\end{aligned}
\label{full_system_22A}
\end{equation}
where $\tilde{\ten{T}}$ and $\tilde{\theta}$ are solutions of \eqref{war_brz_u} and \eqref{war_brz_t}, respectively. System \eqref{full_system_22A} is considered with initial and boundary conditions
\begin{equation}
\begin{array}{rcll}
\theta(\cdot,0) &=&  \theta_0 & \qquad \mbox{in } \Omega, \\
\ten{\varepsilon}^{\bf p}(\cdot,0) &=& \ten{\varepsilon}^{\bf p}_0 &\qquad \mbox{in } \Omega,
\\
\vc{u} &=& 0 & \qquad\mbox{on } \partial\Omega\times (0,T), \\
\frac{\partial\theta}{\partial \vc{n}} &=& 0 & \qquad\mbox{on } \partial\Omega\times (0,T), 
\end{array}
\label{in_bou_cond}
\end{equation}
where $\theta_0$ is a difference between given initial value and initial value $\tilde{\theta}_0$ which was used to {\it cut off} the boundary value problem, see \eqref{war_brz_t} in Appendix \ref{sec:trans}.

The approximate system of equations is constructed using the same argumentation as in \cite{1}. We present briefly this result in Appendix \ref{sec:constr}. Thus, for every $k,l\in\mathbb{N}$, we are looking for  
\begin{equation}
\begin{split}
\vc{u}_{k,l} & = \sum_{n=1}^k\alpha_{k,l}^n(t) \vc{w}_n,
 \\
\theta_{k,l} & = \sum_{m=1}^l\beta_{k,l}^m(t) v_m,
 \\
\ten{\varepsilon}^{\bf p}_{k,l} & = \sum_{n=1}^k\gamma_{k,l}^n(t) \ten{\varepsilon}(\vc{w}_n) + 
\sum_{m=1}^l\delta_{k,l}^m(t) \ten{\zeta}_m^k,
\end{split}
\label{eq:postac1}
\end{equation}
where $\{\vc{w}_n\}, \{v_m\}$ are bases for $W_0^{1,2}(\Omega,\mathbb{R}^3)$ and $W^{1,2}(\Omega)$ with homogeneous Neumann boundary conditions, respectively. Moreover, let us define $V_k^s:=V_k\cap H^s(\Omega,\mathcal{S}^3)$, for $\frac{3}{2}<s\le 2$ and $V_k:= (\mbox{span}\{\ten{\varepsilon}(\ten{w}_1),...,\ten{\varepsilon}(\ten{w}_k)\})^\bot$. Then we denote by $\{\ten{\zeta}^k_n\}_{n=1}^{\infty}$ the basis of $V_k^s$. For more details we refer the reader to Appendix \ref{sec:constr}.

The triple $(\vc{u}_{k,l}, \theta_{k,l}, \ten{\varepsilon}^{\bf p}_{k,l} )$ is a solution to approximate system of equations
\begin{equation}
\begin{array}{rll}
\int_{\Omega}  (\ten{T}_{k,l} - \alpha \ten{I}) : \ten{\varepsilon}(\vc{w}_n) \dx &= 0
& n=1,...,k ,
\\[1ex]
\ten{T}_{k,l} &= \ten{D}(\ten{\varepsilon}(\vc{u}_{k,l}) - \ten{\varepsilon}^{\bf p}_{k,l} ),
\\[1ex]
\int_{\Omega}(\ten{\varepsilon}^{\bf p}_{k,l})_t : \ten{D}\ten{\varepsilon}(\vc{w}_n) \dx &= 
\int_{\Omega}\ten{G}(\tilde{\theta} + \theta_{k,l} ,  \tilde{\ten{T}}^d + \ten{T}^d_{k,l}  ) : \ten{D}\ten{\varepsilon}(\vc{w}_n) \dx 
& n=1,...,k ,
\\[1ex]
\int_{\Omega}(\ten{\varepsilon}^{\bf p}_{k,l})_t : \ten{D}\ten{\zeta}^k_m \dx &= 
\int_{\Omega}\ten{G}(\tilde{\theta} + \theta_{k,l} , \tilde{\ten{T}}^d +  \ten{T}^d_{k,l}   ) : \ten{D}\ten{\zeta}^k_m \dx 
& m=1,...,l ,
\\[1ex]
\int_{\Omega}(\theta_{k,l})_t v_m\dx  + \int_{\Omega}\nabla\theta_{k,l}\cdot\nabla v_m \dx & +\int_{\Omega} \alpha {\div} (\vc{u}_{k,l})_t v_m
\\[1ex]
= \int_{\Omega} \mathcal{T}_k(  (\tilde{\ten{T}}^d + \ten{T}_{k,l}^d  ): & \ten{G}(\tilde{\theta} + \theta_{k,l} ,  \tilde{\ten{T}}^d + \ten{T}^d_{k,l} ) ) v_m \dx & m=1,...,l 
\end{array}
\label{app_system}
\end{equation}
for a.a. $t\in [0,T]$. By $\mathcal{T}_k(\cdot)$ we denoted truncation on level $k$, for definition see Appendix \ref{sec:constr}. For each of approximate equations (for each $k,l\in\mathbb{N}$) we have the initial conditions in the following form 
\begin{equation}
\begin{array}{rclc}
\left( \theta_{k,l}(x,0), v_m\right) &=& \left( \mathcal{T}_k(\theta_0),v_m \right) & m=1,..,l, \\
\left( \ten{\varepsilon}^{\bf p}_{k,l}(x,0), \ten{\varepsilon}(\vc{w}_n) \right)_{\ten{D}} &=& \left(\ten{\varepsilon}^{\bf p}_0, \ten{\varepsilon}(\vc{w}_n) \right)_{\ten{D}}
& n=1,..,k,
\\
\left( \ten{\varepsilon}^{\bf p}_{k,l}(x,0), \ten{\zeta}_m^k) \right)_{\ten{D}} &=& \left(\ten{\varepsilon}^{\bf p}_0, \ten{\zeta}^k_m \right)_{\ten{D}}
& m=1,...,l,
\end{array}
\label{eq:warunki_pocz_app}
\end{equation}
where $\big(\cdot,\cdot\big)$ denotes the inner product in $L^2(\Omega)$ and $\big(\cdot,\cdot\big)_{\ten{D}}$ the inner product in $L^2(\Omega,\mathcal{S}^3)$.

The selection of Galerkin bases and representation of the approximate solutions \eqref{eq:postac1} leads to
\begin{equation}
\lambda_n\alpha_{k,l}^n(t)-\lambda_n 
\gamma_{k,l}^n(t)-\alpha\int_{\Omega}{\div}(\vc{w}_n) \dx=0\,.
\end{equation}
where $\lambda_n$ is a corresponding eigenvalue to $\vc{w}_n$. Notice that the last integral on the left-hand side of above mentioned equation is equal to zero, therefore
\begin{equation}
\alpha_{k,l}^n(t)=
\gamma_{k,l}^n(t)\qquad \mbox{for } n=1,\dots,k\,.
\end{equation}
Let us define 
\begin{equation}
\vc{\xi}(t) = (\beta_{k,l}^1(t),...,\beta_{k,l}^l(t),\gamma_{k,l}^1(t),..., \gamma_{k,l}^k(t),\delta_{k,l}^1(t),...,\delta_{k,l}^l(t) )^T .
\nonumber
\end{equation} 
Moreover, for $m=1,...,l$,
\begin{equation}
\begin{split}
&(\beta_{k,l}^m(t))_t+\mu_m \beta_{k,l}^m(t)+\alpha\sum_{n=1}^k(\alpha_{k,l}^n(t))_t\int_{\Omega}{\div}(\vc{w}_n)v_m \dx
\\
&= \int_{\Omega} \mathcal{T}_k\Big(  \big((\tilde{\ten{T}}^d + \ten{D}\sum_{n=1}^k\alpha_{k,l}^n\ten{\varepsilon}(\vc{w}_n) - \ten{D}(\sum_{n=1}^k\gamma_{k,l}^n(t) \ten{\varepsilon}(\vc{w}_n) +
\sum_{n=1}^l
\delta_{k,l}^n(t) \ten{\zeta}_n ))^d   \big):\tilde{\ten{G}}(x,t,\vc{\xi}(t)) \Big) v_m \dx,
\end{split}
\label{app_system20aa}
\end{equation}
where $\mu_m$ is a corresponding eigenvalue to $v_m$ and 
\begin{equation}
\begin{split}
 \tilde{\ten{G}}(x,t,\vc{\xi}(t))
:& = \ten{G}(\tilde{\theta} + \theta_{k,l},\tilde{\ten{T}}^d + \ten{T}_{k,l}^d )
\\
&=\ten{G}\Big(\tilde{\theta} + \sum_{j=1}^l \beta_{k,l}^j(t) v_j(x) , \tilde{\ten{T}}^d -  \Big(
 \ten{D}\sum_{j=1}^l \delta_{k,l}^j(t) \ten{\zeta}_j^k   \Big)^d   \Big)\,.
\end{split}
\nonumber
\end{equation}
We also observe that
\begin{equation}
(\gamma_{k,l}^n(t))_t  = 
\frac{1}{\lambda_n} 
\int_{\Omega}\tilde{\ten{G}}(x,t,\vc{\xi}(t)) : \ten{D}\ten{\varepsilon}(\vc{w}_n) \dx.
\end{equation}
Thus, we obtain 
\begin{equation}
\begin{aligned}
(\gamma_{k,l}^n(t))_t  &= 
\frac{1}{\lambda_n} 
\int_{\Omega}\tilde{\ten{G}}(x,t,\vc{\xi}(t)) : \ten{D}\ten{\varepsilon}(\vc{w}_n) \dx  ,
\\
(\delta_{k,l}^m(t))_t  &=
\int_{\Omega}\tilde{\ten{G}}(x,t,\vc{\xi}(t)) : \ten{D}\ten{\zeta}_m^k\dx ,
\\
(\beta_{k,l}^m(t))_t &= \int_{\Omega} \mathcal{T}_k\Big(  \big( \tilde{\ten{T}}^d - \ten{D}(
\sum_{n=1}^l
\delta_{k,l}^n(t) \ten{\zeta}_n )^d  \big)
:\tilde{\ten{G}}(x,t,\vc{\xi}(t)) \Big) v_m \dx  - \mu_m \beta_{k,l}^m(t)\\
&\qquad -\alpha\sum_{n=1}^k\frac{1}{\lambda_n} 
\int_{\Omega}\tilde{\ten{G}}(x,t,\vc{\xi}(t)) : \ten{D}\ten{\varepsilon}(\vc{w}_n) \dx
\int_{\Omega}{\div}(\vc{w}_n)v_m \dx.
\end{aligned}
\label{app_system20}
\end{equation}

System \eqref{app_system20} with initial conditions \eqref{eq:warunki_pocz_app} can be equivalently written as the initial value problem 
\begin{equation}\label{47}
\begin{split}
&\frac{d\vc{\xi} }{dt}  = \vc{F}(\vc{\xi}(t),t),
\qquad
t\in [0,T),
\\
&\vc{\xi}(0) =\vc{\xi}_{0}.
\end{split}
\end{equation}
Note that function $\vc{F}(\cdot,\cdot)$ is measurable with respect to $t$, continuous with respect to $\vc{\xi}$ and for every $t$ function $\vc{F}(\cdot,t)$ is bounded. Let us fix $k,l\in\mathbb{N}$. According to Carath\'eodory theorem, see \cite[Theorem 3.4, Appendix]{maleknecas} or \cite[Appendix $(61)$]{zeidlerB}, there exist absolutely continuous functions $\beta_{k,l}^m(t)$, $\gamma_{k,l}^n(t)$ and $\delta_{k,l}^m(t)$ for every $n \leq k$ and $m \leq l$ on some time interval $[0,t^*]$. Moreover, for every $n \leq k$, there exists an absolutely continuous function $\alpha_{k,l}^n(t)$ on $[0,t^*]$.

\subsection{Boundedness of approximate solutions and limit passage with $l\to\infty$.}
\label{sec:boun1l}

In this section we prove uniform boundedness of approximate solutions. Some of them are uniform with respect to both approximation parameters and some of them are uniform only with respect to $l$. Since the first limit passage is done with $l$ going to $\infty$, we focus here on bounds which are uniform with respect to that parameter.

Due to the fact that some of the following estimates go similarly to ones presented in \cite{1}, we skip them. For those cases we only underline the difference between previous results and the following ones. In comparison to model without thermal expansion, the estimates regarding to $\{(\vc{u}_{k,l})_t\}$ appear, for which we present complete proofs.

\begin{defi}
We say that $\mathcal{E}$
is the potential energy if
\begin{equation}
\mathcal{E}(\ten{\varepsilon}(\vc{u}),\ten{\varepsilon}^{\bf p}): =  \frac{1}{2}\int_{\Omega}\ten{D}(\ten{\varepsilon}(\vc{u}) - \ten{\varepsilon}^{\bf p}):(\ten{\varepsilon}(\vc{u}) - \ten{\varepsilon}^{\bf p} ) \dx .
\nonumber
\end{equation}
\label{energia}
\end{defi}

\begin{lemat}
There exists a constant $C$ which
is uniform with respect to  $k$ and $l$ such that
\begin{equation}
\sup_{t\in [0,T]} \mathcal{E}(\ten{\varepsilon}(\vc{u}_{k,l}) , \ten{\varepsilon}^{\bf p}_{k,l}) (t)
+ c \| \tilde{\ten{T}}^d + \ten{T}_{k,l}^d\|^p_{L^p(0,T,L^p(\Omega))}
\leq
C.
\label{osz2}
\end{equation}
\label{pom_2}
\end{lemat}
The idea of this proof is the same as of \cite[Lemma 3.2]{1}. The only difference is that term $\int_{\Omega}\alpha\ten{I}(\ten{\varepsilon}(\vc{u}_{k,l}))_t\dx$ appears. After integration by parts it is equal to zero and has no influence of final result.

\begin{uwaga} 
From \eqref{osz2} we immediately observe that the sequence $\{\ten{T}_{k,l}^d\}$ is uniformly bounded in the space $L^p(0,T,L^p(\Omega,\mathcal{S}^3))$ with respect to $k$ and $l$.
Additionally, using the growth conditions on function $\ten{G}$, see Assumptions \ref{ass_G}, we conclude the uniform boundedness of the sequence 
$\{\ten{G}(\tilde{\theta} + \theta_{k,l},\tilde{\ten{T}}^d + \ten{T}_{k,l}^d )\}$ in the space $L^{p'}(0,T,L^{p'}(\Omega,\mathcal{S}^3))$. Thus, we obtain the uniform boundedness of the sequence $\{(\tilde{\ten{T}}^d + \ten{T}_{k,l}^d ):\ten{G}(\tilde{\theta} + \theta_{k,l},\tilde{\ten{T}}^d + \ten{T}_{k,l}^d )\}$ in $L^1(0,T,L^1(\Omega))$.
\label{wsp_ogr_T}
\end{uwaga}

\begin{lemat}
For every fixed $k$ the sequence $\{(\ten{\varepsilon}^{\bf p}_{k,l})_t\}$ is uniformly bounded in $L^{p'}(0,T,(H^{s}(\Omega,\mathcal{S}^3))')$ with respect to $l$.
\label{wsp_org_epa}
\end{lemat}

Proof of Lemma \ref{wsp_org_epa} is the most tricky one in this part of the paper. We refer the reader to Appendix \ref{sec:Lemma}, where this proof is presented. 

Let us calculate time derivative of equation \eqref{app_system}$_{(1)}$. We obtain
\begin{equation}
\int_{\Omega} (\ten{T}_{k,l} -\alpha\ten{I})_t: \ten{\varepsilon}(\vc{w}_{n}) \dx + 
\int_{\Omega} (\ten{T}_{k,l} -\alpha\ten{I}): \ten{\varepsilon}(\vc{w}_{n})_t \dx 
=0.
\label{eq:11}
\end{equation}
Since $\alpha \ten{I}$ is constant and $\vc{w}_{n}$ does not depend on time, then \eqref{eq:11} is equivalent to
\begin{equation}
\int_{\Omega} (\ten{T}_{k,l})_t: \ten{\varepsilon}(\vc{w}_{n}) \dx =0.
\label{eq:11a}
\end{equation}
Thus
\begin{equation}
\int_{\Omega} \ten{D}(\ten{\varepsilon}(\vc{u}_{k,l}))_t: \ten{\varepsilon}(\vc{w}_{n}) \dx =
\int_{\Omega} \ten{D}(\ten{\varepsilon}^{\bf p}_{k,l})_t:
\ten{\varepsilon}(\vc{w}_{n}) \dx ,
\label{eq:11aa}
\end{equation}
for a.a. $t\in (0,T)$.

\begin{lemat}
For every fixed $k\in\mathbb{N}$ the sequence $\{(\vc{u}_{k,l})_t \}$ is uniformly bounded in $L^{p'}(0,T,W^{1,2}_0(\Omega,\mathbb{R}^3))$ with respect to~$l$.
\label{lm:2.6}
\end{lemat}
\begin{proof}
Let us multiply the equation \eqref{eq:11aa} by $(\alpha^n_{k,l}(t))_t$ and sum over $k\leq n$. Then we obtain
\begin{equation}
\begin{split}
c\|\ten{\varepsilon}(\vc{u}_{k,l}))_t\|_{L^2(\Omega)}^2 & \leq \int_{\Omega} \ten{D}(\ten{\varepsilon}(\vc{u}_{k,l}))_t: (\ten{\varepsilon}(\vc{u}_{k,l}))_t \dx =
\int_{\Omega} \ten{D}(\ten{\varepsilon}^{\bf p}_{k,l})_t:
(\ten{\varepsilon}(\vc{u}_{k,l}))_t \dx 
\\
& \leq d \| (\ten{\varepsilon}^{\bf p}_{k,l})_t \|_{(H^s(\Omega))'} 
\|(\ten{\varepsilon}(\vc{u}_{k,l}))_t \|_{H^s(\Omega)}.
\end{split}
\label{eq:11b}
\end{equation}
Since for each $k$ function $\vc{u}_{k,l}$ is a finite dimensional function then we may estimate $H^s(\Omega)$-norm by  $L^2(\Omega)$ norm and we obtain
\begin{equation}
c\|\ten{\varepsilon}(\vc{u}_{k,l}))_t\|_{L^2(\Omega)}^2  \leq
 d_2 \| (\ten{\varepsilon}^{\bf p}_{k,l})_t \|_{(H^s(\Omega))'} 
\|(\ten{\varepsilon}(\vc{u}_{k,l}))_t \|_{L^2(\Omega)},
\end{equation}
which provides to 
\begin{equation}
c\|\ten{\varepsilon}(\vc{u}_{k,l}))_t\|_{L^2(\Omega)}  \leq
 d_2 \| (\ten{\varepsilon}^{\bf p}_{k,l})_t \|_{(H^s(\Omega))'} .
\end{equation}
Since we consider homogeneous Dirichlet boundary-value problem for displacement, we use Poincar\'{e}'s inequality. Then integrating over time interval $(0,T)$ and using Lemma \ref{wsp_org_epa} we finish the proof. 
\end{proof}

\begin{lemat}
There exists a constant $C$, depending on the domain $\Omega$ and the time interval $(0,T)$, such that for every $k\in\mathbb{N}$
\begin{equation}
\begin{split}
\sup_{0\leq t\leq T}\|\theta_{k,l}(t)\|^2_{L^2(\Omega)} &+
\|\theta_{k,l}\|^2_{L^2(0,T,W^{1,2}(\Omega))} +
\|(\theta_{k,l})_t\|^2_{L^2(0,T,W^{-1,2}(\Omega))}
\\
& \leq C\Big(\|\mathcal{T}_k\Big( ( \tilde{\ten{T}}^d + \ten{T}_{k,l}^d ):\ten{G}(\tilde{\theta} + \theta_{k,l},\tilde{\ten{T}}^d + \ten{T}_{k,l}^d) \Big)\|^2_{L^2(0,T,L^2(\Omega))}
\\
& \qquad \quad 
+
\|\mathcal{T}_k(\theta_0)\|_{L^2(\Omega)}^2
+\|(\ten{\varepsilon}(\vc{u}_{k,l}))_t\|_{L^{p'}(0,T,L^2(\Omega))}^2
\Big).
\label{numer}
\end{split}
\end{equation}
\label{lm:7}
\end{lemat}

\begin{proof}
The proof follows from  the standard tools for parabolic equations, see e.g. Evans \cite{Evans}. The only problem which appears is that one has to estimate term $\int_{\Omega} \alpha \ten{I} :(\ten{\varepsilon}(\vc{u}_{k,l}))_t \theta_{k,l}\dx$, which is not trivial. 
\begin{equation}
\begin{split}
\alpha \int_0^T\int_\Omega \theta_{k,l}\div (\vc{u}_{k,l})_t \dx\dt 
& \leq \alpha \int_0^T\|\theta_{k,l}\|_{L^2(\Omega)} \|\div (\vc{u}_{k,l})_t\|_{L^2(\Omega)}\dt
\\
& \leq \alpha \|\theta_{k,l}\|_{L^\infty(0,T,L^2(\Omega))} \|\div ( \vc{u}_{k,l})_t\|_{L^1(0,T,L^2(\Omega))}
\\
& \leq \epsilon \|\theta_{k,l}\|_{L^\infty(0,T,L^2(\Omega))}^2 + C(\epsilon) \alpha^2 \|\div (\vc{u}_{k,l})_t\|_{L^{p'}(0,T,L^2(\Omega))}^2
\\
& \leq \epsilon \|\theta_{k,l}\|_{L^\infty(0,T,L^2(\Omega))}^2 + C(\epsilon) \alpha^2 \|(\ten{\varepsilon}(\vc{u}_{k,l}))_t\|_{L^{p'}(0,T,L^2(\Omega))}^2.
\end{split}
\end{equation}
Putting the first term from right-hand side of above mentioned inequality into left-hand side we complete the proof.
\end{proof}

\begin{uwaga}
The uniform boundedness of solutions 
implies the global existence of approximate solutions, i.e. existence of solutions $\{\alpha_{k,l}^n(t),\beta_{k,l}^m(t),\gamma_{k,l}^n(t),\delta_{k,l}^m(t)\}$ on the whole time interval $[0,T]$ for each $n=1,...,k$ and $m=1,...,l$.
\end{uwaga}

\begin{uwaga}
Testing the heat equation by $1$ we obtain that $\sup_{t\in [0,T]}\int_{\Omega}\theta_{k,l}(t)\dx $ is uniformly bounded with respect to both parameters. Since we consider quasi-static problem, only thermal and potential energy of material are taken into account in total energy. Using Lemma \ref{pom_2} we obtain that total physical energy is finite. 
\end{uwaga}

Now, let us multiply equations from system \eqref{app_system} by smooth time-dependent functions and let us integrate they over $[0,T]$. Then 
\begin{equation}\label{58}
\begin{split}
\int_0^T\int_{\Omega}(\ten{T}_{k,l} - \alpha\ten{I}):\nabla\vc{w}_n  \varphi_1(t) \dxdt 
&= 0,
\\
\int_0^T\int_{\Omega}(\ten{\varepsilon}^{\bf p}_{k,l})_t : \ten{D}\ten{\varepsilon}(\vc{w}_n) \varphi_2(t)\dxdt 
& = 
\int_0^T\int_{\Omega}\ten{G}(\tilde{\theta} + \theta_{k,l},   \tilde{\ten{T}}^d + \ten{T}^d_{k,l}  ) : \ten{D}\ten{\varepsilon}(\vc{w}_n) 
\varphi_2(t)\dxdt, 
\end{split}
\end{equation}
for $n=1,...,k$ and
\begin{equation}\label{60}
\begin{split}
\int_0^T& \int_{\Omega}(\ten{\varepsilon}^{\bf p}_{k,l})_t : \ten{D}\ten{\zeta}^k_m \varphi_3(t)\dxdt = 
\int_0^T\int_{\Omega}\ten{G}(\tilde{\theta} + \theta_{k,l} , \tilde{\ten{T}}^d +\ten{T}^d_{k,l} ) : \ten{D}\ten{\zeta}^k_m \varphi_3(t)\dxdt, 
\\
-\int_0^T&\int_{\Omega} \theta_{k,l}\varphi_4'(t) v_m \dxdt -
\int_{\Omega} \theta_0(x)\varphi_4(0) v_m \dx   + 
\int_0^T\int_{\Omega}  \nabla\theta_{k,l}\cdot \nabla v_m \varphi_4(t) \dxdt 
\\ & 
+ \int_0^T\int_{\Omega} \alpha\mbox{div}(\vc{u}_{k,l})_t v_m \varphi_4(t)\dxdt = 
\int_0^T\int_{\Omega} \mathcal{T}_k\left((\tilde{\ten{T}}^d + \ten{T}^d_{k,l}):\ten{G}(\tilde{\theta} + \theta_{k,l},\tilde{\ten{T}}^d + \ten{T}^d_{k,l})\right)\varphi_4(t)v_m \dxdt,
\end{split}
\end{equation}
holds for every test functions ${\varphi}_1, {\varphi}_2, {\varphi}_3\in C^{\infty}([0,T])$ and $\varphi_4\in C^{\infty}_c([-\infty,T))$. 

From the uniform boundedness of approximate solutions with respect to $l$ we obtain that there exist at least subsequences but still denoted by the index $l$ that the following convergences hold
\begin{equation}
\begin{array}{cl}
\ten{T}_{k,l}\rightharpoonup \ten{T}_k  &  \mbox{weakly in }   L^2(0,T,L^2(\Omega,\mathcal{S}^3)),\\
\ten{T}^d_{k,l}\rightharpoonup \ten{T}^d_k  &  \mbox{weakly in }   L^p(0,T,L^p(\Omega,\mathcal{S}^3_d)),\\
\ten{G}(\tilde{\theta} + \theta_{k,l},\tilde{\ten{T}}^d + \ten{T}_{k,l}^d)\rightharpoonup \ten{\chi}_k  & \mbox{weakly in } L^{p'}(0,T,L^{p'}(\Omega,\mathcal{S}^3_d)), \\
\theta_{k,l}\rightharpoonup \theta_k  &  \mbox{weakly in }   L^2(0,T,W^{1,2}(\Omega)),\\
\theta_{k,l}\rightarrow \theta_k  &  \mbox{a.e. in } \Omega \times (0,T),\\
(\ten{\varepsilon}^{\bf p}_{k,l})_t\rightharpoonup(\ten{\varepsilon}^{\bf p}_{k})_t&  \mbox{weakly in }   
L^{p'}(0,T,(H^s(\Omega,\mathcal{S}^3))'),
\\
(\vc{u}_{k,l})_t \rightharpoonup (\vc{u}_{k})_t & \mbox{weakly in } L^{p'}(0,T, W^{1,2}_0(\Omega,\mathbb{R}^3))
\end{array}
\end{equation}
Passing to the limit in \eqref{58} and \eqref{60}$_{(1)}$ yields
\begin{equation}\label{limit1}
\begin{split}
\int_0^T\int_{\Omega}(\ten{T}_{k}-\alpha\ten{I}):\nabla\vc{w}_n  \varphi_1(t) \dxdt 
&= 0, \quad n=1,\dots, k 
\end{split}
\end{equation}
\begin{equation}\label{limit2}\begin{split}
\int_0^T\int_{\Omega}(\ten{\varepsilon}^{\bf p}_{k})_t : \ten{D}\ten{\varepsilon}(\vc{w}_n) \varphi_2(t)\dxdt = 
\int_0^T\int_{\Omega}\ten{\chi}_{k}: \ten{D}\ten{\varepsilon}(\vc{w}_n) 
\varphi_2(t)\dxdt, 
\quad n=1,...,k ,
\\
\int_0^T\int_{\Omega}(\ten{\varepsilon}^{\bf p}_{k})_t : \ten{D}\ten{\zeta}^k_m \varphi_3(t)\dxdt = 
\int_0^T\int_{\Omega}\ten{\chi}_{k}  : \ten{D}\ten{\zeta}^k_m \varphi_3(t)\dxdt, 
\quad  m\in\mathbb{N},
\end{split}\end{equation}
holds for every test functions ${\varphi}_1, {\varphi}_2, {\varphi}_3\in C^{\infty}([0,T])$. By construction of bases set $(\{\ten{\varepsilon}(\vc{w}_n)\}_{n=1}^k,\{\ten{\zeta}_m^k\}_{m=1}^{\infty} )$ is a dense set in $L^{p}(\Omega,\mathcal{S}^3)$. Thus
\begin{equation}\label{65}
\int_0^T\int_{\Omega}(\ten{\varepsilon}^{\bf p}_{k})_t : \ten{\varphi}\dxdt = 
\int_0^T\int_{\Omega}\ten{\chi}_{k}  :  \ten{\varphi}\dxdt,
\end{equation}
holds for all $ \ten{\varphi}\in C^\infty([0,T],L^{p}(\Omega,\mathcal{S}^3))$ and then also for all $ \ten{\varphi}\in L^{p}(0,T;L^{p}(\Omega,\mathcal{S}^3))$.

The last part of this section is devoted to identification the weak limit of the nonlinear term $\ten{\chi}_k$ and showing the convergence of $\ten{G}(\tilde{\theta} + \theta_{k,l},\tilde{\ten{T}}^d + \ten{T}_{k,l}^d):(\tilde{\ten{T}}^d+\ten{T}^d_{k,l})$. At this moment the limit of $\{\ten{G}(\tilde{\theta} + \theta_{k,l},\tilde{\ten{T}}^d + \ten{T}_{k,l}^d):(\tilde{\ten{T}}^d+\ten{T}^d_{k,l})\}$ is not defined, because $\{\ten{G}(\tilde{\theta} + \theta_{k,l},\tilde{\ten{T}}^d + \ten{T}_{k,l}^d)\}$ and $\{\tilde{\ten{T}}^d+\ten{T}^d_{k,l}\}$ converges only weakly.

\begin{lemat}
The following inequality holds for solutions of approximate system
\begin{equation}
\limsup_{l\rightarrow\infty}\int_{0}^{t}\int_{\Omega}\ten{G}(\tilde{\theta} + \theta_{k,l},\tilde{\ten{T}}^d + \ten{T}^d_{k,l}):\ten{T}^d_{k,l} \dxdt \leq
\int_{0}^{t}\int_{\Omega}\ten{\chi}_k:\ten{T}^d_k \dxdt .
\label{teza-8}
\end{equation}
\label{lm:8}
\end{lemat}

\begin{proof}
For each $\mu>0, t_2\le T-\mu, t\ge0, $ let $\psi_\mu:\mathbb{R}_+\to\mathbb{R}_+$ be defined as follows
\begin{equation}\label{psi-mu}
\psi_{\mu,t_2}(t)=\left\{
\begin{array}{lcl}
1&{\rm for}&t\in[0,t_2),\\
-\frac{1}{\mu}t+\frac{1}{\mu}t_2+1&{\rm for} & t\in[t_2, t_2+\mu),\\
0&{\rm for} & t\ge t_2+\mu.
\end{array}\right.
\end{equation}
The potential energy is an absolutely continuous function and calculating the time derivative of $\mathcal{E}(t)$ we get for a.a. $t\in[0,T]$
\begin{equation}
\begin{split}
\frac{d}{dt} \mathcal{E}(\ten{\varepsilon}(\vc{u}_{k,l}) , \ten{\varepsilon}^{\bf p}_{k,l}) 
& = 
\int_{\Omega}\ten{D}(\ten{\varepsilon}(\vc{u}_{k,l}) - \ten{\varepsilon}^{\bf p}_{k,l}):(\ten{\varepsilon}(\vc{u}_{k,l}))_t 
\dx 
\\
& \quad
-
\int_{\Omega}\ten{D}(\ten{\varepsilon}(\vc{u}_{k,l}) - \ten{\varepsilon}^{\bf p}_{k,l}): (\ten{\varepsilon}^{\bf p}_{k,l})_t \dx.
\end{split}
\label{pochodna}\end{equation}
In the first step we  multiply \eqref{app_system}$_{(1)}$ by $\{(\alpha_{k,l}^n)_t\}$ for each $n\leq k$.
Summing over $n=1,...,k$ we obtain
\begin{equation}
 \int_{\Omega}\left(\ten{D}(\ten{\varepsilon}(\vc{u}_{k,l}) - \ten{\varepsilon}^{\bf p}_{k,l}) -\alpha \ten{I}\right): (\ten{\varepsilon}(\vc{u}_{k,l}))_t \dx =
0.
\label{pierwsze_r}
\end{equation}
Hence
\begin{equation}
\int_{\Omega}\ten{D}(\ten{\varepsilon}(\vc{u}_{k,l}) - \ten{\varepsilon}^{\bf p}_{k,l}) : (\ten{\varepsilon}(\vc{u}_{k,l}))_t \dx -
\int_{\Omega} \alpha \ten{I}: (\ten{\varepsilon}(\vc{u}_{k,l}))_t \dx=0.
\label{pierwsze_r2aaaa}
\end{equation}
Integrating by parts second integral we observe that it is equal to zero. In the second step we  multiply \eqref{app_system}$_{(4)}$ by $\delta^m_{k,l}$ and summing over $m=1,...,l$, we obtain the identity, which is equivalent to
\begin{equation}
\int_{\Omega}(\ten{\varepsilon}^{\bf p}_{k,l})_t:\ten{T}_{k,l} \dx=
\int_{\Omega}\ten{G}(\tilde{\theta} + \theta_{k,l},\tilde{\ten{T}}^d + \ten{T}^d_{k,l}):\ten{T}_{k,l} \dx.
\label{drugie_r}
\end{equation}
Thus
\begin{equation}\label{ene}
\begin{split}
\frac{d}{dt} \mathcal{E}(\ten{\varepsilon}(\vc{u}_{k,l}) , \ten{\varepsilon}^{\bf p}_{k,l}) 
 = &
 -
\int_{\Omega}\ten{G}(\tilde{\theta} + \theta_{k,l},\tilde{\ten{T}}^d + \ten{T}^d_{k,l}):\ten{T}^d_{k,l}\dx.
\end{split}
\end{equation}
Multiplying \eqref{ene} by $\psi_{\mu,t_2}(t)$ and integrate over $(0,T)$
\begin{equation}\label{mu2}
\begin{split}
\int_{0}^{T}
\frac{d}{d\tau} \mathcal{E}(\ten{\varepsilon}(\vc{u}_{k,l}) , \ten{\varepsilon}^{\bf p}_{k,l}) \,\psi_{\mu,t_2}\dt
 = 
-
\int_0^T\int_{\Omega}\ten{G}(\tilde{\theta} + \theta_{k,l},\tilde{\ten{T}}^d + \ten{T}^d_{k,l}):\ten{T}^d_{k,l} \,\psi_{\mu,t_2}\dxdt.
\end{split}
\end{equation}
Let us now integrate by parts the left hand side of \eqref{mu2}
\begin{equation}\label{mu3}
\begin{split}
\int_{0}^{T}
\frac{d}{d\tau} \mathcal{E}(\ten{\varepsilon}(\vc{u}_{k,l}) , \ten{\varepsilon}^{\bf p}_{k,l}) \,\psi_{\mu,t_2}\dt
=\frac{1}{\mu}\int_{t_2}^{t_2+\mu}\mathcal{E}(\ten{\varepsilon}(\vc{u}_{k,l}(t)) , \ten{\varepsilon}^{\bf p}_{k,l}(t)) \dt-
\mathcal{E}(\ten{\varepsilon}(\vc{u}_{k,l}(0)) , \ten{\varepsilon}^{\bf p}_{k,l}(0)).
\end{split}\end{equation}
Passing to the limit with $l\to\infty$ we obtain
\begin{equation}\label{mu4}
\begin{split}
\liminf\limits_{l\to\infty}\int_{0}^{T}
\frac{d}{d\tau}& \mathcal{E}(\ten{\varepsilon}(\vc{u}_{k,l}) , \ten{\varepsilon}^{\bf p}_{k,l}) \,\psi_{\mu,t_2}\dt\\
&=\liminf\limits_{l\to\infty}\frac{1}{\mu}\int_{t_2}^{t_2+\mu}\mathcal{E}(\ten{\varepsilon}(\vc{u}_{k,l}) , \ten{\varepsilon}^{\bf p}_{k,l}) \dt-
\lim\limits_{l\to\infty}\mathcal{E}(\ten{\varepsilon}(\vc{u}_{k,l}(0)) , \ten{\varepsilon}^{\bf p}_{k,l}(0))\\
&\ge \frac{1}{\mu}\int_{t_2}^{t_2+\mu}\mathcal{E}(\ten{\varepsilon}(\vc{u}_{k}(t)) , \ten{\varepsilon}^{\bf p}_{k}(t)) \dt-
\mathcal{E}(\ten{\varepsilon}(\vc{u}_{k}(0)) , \ten{\varepsilon}^{\bf p}_{k}(0)).
\end{split}
\end{equation}
Note that the last inequality holds due to the weak lower semicontinuity in 
$L^2(0,T,L^2(\Omega;\mathcal{S}^3))$. 
Let us take $t_1,\tau\in (0,T)$ and $\epsilon$ such that $\epsilon < \min(t_1,T-\tau)$. Then we choose in \eqref{limit1} the test functions  $\varphi_1(t)=((\alpha^n_k)_t*\eta_{\epsilon}\mathbf{1}_{(t_1, \tau)})*\eta_{\epsilon}$, and  in \eqref{65} $\ten{\varphi}=(\ten{T}_k^d*\eta_{\epsilon}\mathbf{1}_{(t_1, \tau)})*\eta_{\epsilon}$, where $\eta_\epsilon$ is a standard mollifier and we mollify with respect to time. Thus we obtain
\begin{equation}
\begin{split}
\int_0^T\int_{\Omega}  \ten{T}_{k} : \ten{\varepsilon}(((\alpha^n_k)_t*\eta_{\epsilon}\mathbf{1}_{(t_1,\tau)})*\eta_{\epsilon}\vc{w}_n) \dx &= 0,
\\
\int_0^T\int_{\Omega}(\ten{\varepsilon}^{\bf p}_{k})_t : (\ten{T}_k^d*\eta_{\epsilon}\mathbf{1}_{(t_1,\tau)})*\eta_{\epsilon} \dx 
= 
\int_0^T\int_{\Omega}\ten{\chi}_{k} : &(\ten{T}_k^d*\eta_{\epsilon}\mathbf{1}_{(t_1,\tau)})*\eta_{\epsilon}\dx ,
\end{split}
\label{app_system_n}
\end{equation}
for  $n=1,...,k $. Using the properties of convolution and summing \eqref{app_system_n}$_{(1)}$ over $n=1,...,k$ we obtain
\begin{equation}
\begin{split}
\int_{t_1}^{\tau} \int_{\Omega}\ten{D}\left(\ten{\varepsilon}(\vc{u}_{k}) - \ten{\varepsilon}^{\bf p}_{k}\right)*\eta_{\epsilon}: (\ten{\varepsilon}(\vc{u}_{k})*\eta_{\epsilon})_t \dxdt &=
0,
\\
\int_{t_1}^{\tau}\int_{\Omega}(\ten{\varepsilon}^{\bf p}_{k}*\eta_{\epsilon})_t:\ten{T}_{k}^d*\eta_{\epsilon} \dxdt =
\int_{t_1}^{\tau}\int_{\Omega}\ten{\chi}_{k}*\eta_{\epsilon}:& \ten{T}_{k}^d*\eta_{\epsilon} \dxdt .
\end{split}
\label{drugie_r2}
\end{equation}
Properties of traceless matrices, i.e. $\ten{A}\in\mathcal{S}^3_d$ and $\ten{B}\in\mathcal{S}^3$ then $\ten{A}:\ten{B}^d=\ten{A}:\ten{B}$ holds, allow us to replace deviatoric part of matrix $\ten{T}_k^d$ by $\ten{T}_k$ in \eqref{drugie_r2}$_{(2)}$ and it is still well defined. Subtracting \eqref{drugie_r2}$_{(2)}$ from \eqref{drugie_r2}$_{(1)}$ and passing with $\epsilon\to 0$ we obtain the equality
\begin{equation}
\frac{1}{2}\int_{\Omega}\ten{D}(\ten{\varepsilon}(\vc{u}_k) - \ten{\varepsilon}^{\bf p}_k):(\ten{\varepsilon}(\vc{u}_k) - \ten{\varepsilon}^{\bf p}_k) \dx \Big|_{t_1}^{\tau}
 = - 
\int_{t_1}^{t_2}\int_{\Omega}\ten{\chi}_k:\ten{T}^d_k \dxdt .
\label{granica_l}
\end{equation}
Since $\ten\varepsilon({\vc{u}_k}),  \ten{\varepsilon}^{\bf p}_k\in C_{w}([0,T],L^2(\Omega,\mathcal{S}^3))$, then we may pass
with $t_1\to 0$ and conclude
\begin{equation}\label{gran}
\mathcal{E}(\ten{\varepsilon}(\vc{u}_{k}(\tau)) , \ten{\varepsilon}^{\bf p}_{k}(\tau)) -
\mathcal{E}(\ten{\varepsilon}(\vc{u}_{k}(0)) , \ten{\varepsilon}^{\bf p}_{k}(0))=- 
\int_{0}^{t_2}\int_{\Omega}\ten{\chi}_k:\ten{T}^d_k \dxdt .
\end{equation}
Multiplying \eqref{gran} by  $\frac{1}{\mu}$ and integrating     over the interval $(t_2, t_2+\mu)$ we get
\begin{equation}
\begin{split}
\frac{1}{\mu}\int_{t_2}^{t_2+\mu}\mathcal{E}(\ten{\varepsilon}(\vc{u}_{k}(\tau)) , \ten{\varepsilon}^{\bf p}_{k}(\tau)) \dtau-
\mathcal{E}(\ten{\varepsilon}(\vc{u}_{k}(0)) , \ten{\varepsilon}^{\bf p}_{k}(0))=-
\frac{1}{\mu}\int_{t_2}^{t_2+\mu}\int_{0}^{\tau}\int_{\Omega}\ten{\chi}_k:\ten{T}^d_k \dxdt \dtau.
\end{split}\end{equation}
For brevity we denote 
\begin{equation}
F(s):=\int_{\Omega}\ten{\chi}_k:\ten{T}^d_k\dx,
\nonumber
\end{equation}
which is obviously in $L^1(0,T)$. Then we may apply the Fubini theorem
\begin{equation}\label{71}\begin{split}
\frac{1}{\mu}\int_{t_2}^{t_2+\mu}\int_0^\tau F(s) \ds\dtau&=\frac{1}{\mu}\int_{\mathbb{R}^2} 
\mathbf{1}_{\{0\le s\le \tau\}}(s)\mathbf{1}_{\{t_2\le \tau\le t_2+\mu\}}(\tau) F(s)\ds\dtau\\
&=\frac{1}{\mu}\int_\mathbb{R} \left(\int_\mathbb{R}
\mathbf{1}_{\{0\le s\le\tau\}}(s)\mathbf{1}_{\{t_2\le \tau\le t_2+\mu\}} (\tau)\dtau\right)F(s)\ds.
\end{split}
\end{equation}
The crucial  observation is that
\begin{equation}\label{psi}
\psi_{\mu,t_2}(s)=\frac{1}{\mu}\int_\mathbb{R}
\mathbf{1}_{\{0\le s\le\tau\}}(s)\mathbf{1}_{\{t_2\le \tau\le t_2+\mu\}} (\tau)\dtau.
\end{equation}
Hence using \eqref{mu2} and \eqref{mu4} we conclude
\begin{equation}
-\int_{0}^{T}\int_{\Omega}\ten{\chi}_k:\ten{T}^d_k \, \psi_{\mu,t_2}\dxdt \le
\liminf\limits_{l\to\infty}\left( 
-
\int_0^T\int_{\Omega}\ten{G}(\tilde{\theta} + \theta_{k,l},\tilde{\ten{T}}^d + \ten{T}^d_{k,l}):\ten{T}^d_{k,l} \,\psi_{\mu,t_2}\dxdt \right),
\end{equation}
which is nothing else than
\begin{equation}
\limsup_{l\rightarrow\infty}\int_{0}^{T}\int_{\Omega}\ten{G}(\tilde{\theta} + \theta_{k,l},\tilde{\ten{T}}^d + \ten{T}^d_{k,l}):\ten{T}^d_{k,l} \ \psi_{\mu,t_2}\dxdt \leq
\int_{0}^{T}\int_{\Omega}\ten{\chi}_k:\ten{T}^d_k \ \psi_{\mu,t_2}\dxdt .
\label{jedna_nierownosc}
\end{equation}
Let us observe now that
\begin{equation}
\begin{split}
\limsup\limits_{l\to\infty}&\int_{0}^{t_2}\int_{\Omega}\ten{G}(\tilde{\theta} + \theta_{k,l},\tilde{\ten{T}}^d + \ten{T}^d_{k,l}):\ten{T}^d_{k,l} \dxdt\\
&\le
\limsup\limits_{l\to\infty}\int_{0}^{t_2}\int_{\Omega}\ten{G}(\tilde{\theta} + \theta_{k,l},\tilde{\ten{T}}^d + \ten{T}^d_{k,l}):(\tilde{\ten{T}}^d + \ten{T}^d_{k,l}) \dxdt\\
&-\lim\limits_{l\to\infty}\int_{0}^{t_2}\int_{\Omega}\ten{G}(\tilde{\theta} + \theta_{k,l},\tilde{\ten{T}}^d + \ten{T}^d_{k,l}):\tilde{\ten{T}}^d  \dxdt\\
&\le
\limsup\limits_{l\to\infty}\int_{0}^{t_2+\mu}\int_{\Omega}\ten{G}(\tilde{\theta} + \theta_{k,l},\tilde{\ten{T}}^d + \ten{T}^d_{k,l}):(\tilde{\ten{T}}^d + \ten{T}^d_{k,l})\psi_{\mu,t_2} \dxdt\\
&-\lim\limits_{l\to\infty}\int_{0}^{t_2}\int_{\Omega}\ten{G}(\tilde{\theta} + \theta_{k,l},\tilde{\ten{T}}^d + \ten{T}^d_{k,l}):\tilde{\ten{T}}^d  \dxdt\\
&\le
\limsup_{l\rightarrow\infty}\int_{0}^{t_2+\mu}\int_{\Omega}\ten{G}(\tilde{\theta} + \theta_{k,l},\tilde{\ten{T}}^d + \ten{T}^d_{k,l}):\ten{T}^d_{k,l} \ \psi_{\mu,t_2}\dxdt\\
&+\lim\limits_{l\to\infty}\int_{0}^{t_2+\mu}\int_{\Omega}\ten{G}(\tilde{\theta} + \theta_{k,l},\tilde{\ten{T}}^d + \ten{T}^d_{k,l}):\tilde{\ten{T}}^d \ \psi_{\mu,t_2}\dxdt\\
&-\lim\limits_{l\to\infty}\int_{0}^{t_2}\int_{\Omega}\ten{G}(\tilde{\theta} + \theta_{k,l},\tilde{\ten{T}}^d + \ten{T}^d_{k,l}):\tilde{\ten{T}}^d  \dxdt\\
& \leq
\int_{0}^{t_2+\mu}\int_{\Omega}\ten{\chi}_k:\ten{T}^d_k \ \psi_{\mu,t_2}\dxdt
+\lim\limits_{l\to\infty}
\int_{t_2}^{t_2+\mu}\int_{\Omega}\ten{G}(\tilde{\theta} + \theta_{k,l},\tilde{\ten{T}}^d + \ten{T}^d_{k,l}):\tilde{\ten{T}}^d  \psi_{\mu,t_2}\dxdt.
\end{split}
\end{equation}
Passing with $\mu\to0$ yields \eqref{teza-8}. The proof is complete.
\end{proof}

To identify the weak limit~$\ten{\chi}_k$ we use the Minty-Browder trick. This procedure was presented in \cite{1}. We use monotonicity of function $\ten{G}(\cdot,\cdot)$ with respect to second variable and pointwise convergence of temperature $\{\theta_{k,l}\}_{k=1}^{\infty}$ to obtain that
\begin{equation}
\ten{\chi}_k =\ten{G}(\tilde{\theta} + \theta_{k},\tilde{\ten{T}}^d + \ten{T}_k^d )\quad \mbox{a.e. in}\ (0,T)\times\Omega.
\label{82}
\end{equation}
This implies that for every $k\in\mathbb{N}$
\begin{equation}
\ten{G}(\tilde{\theta} + \theta_{k,l},\tilde{\ten{T}}^d + \ten{T}_{k,l}^d)\rightharpoonup \ten{G}(\tilde{\theta} + \theta_k,\tilde{\ten{T}}^d + \ten{T}_k^d)\quad \mbox{in} \ L^{p'}(0,T,L^{p'}(\Omega,\mathcal{S}^3)),
\end{equation}
as $l\rightarrow \infty$. Moreover, using monotonicity of function $\ten{G}(\cdot,\cdot)$ and pointwise convergence of temperature $\{\theta_{k,l}\}_{k=1}^{\infty}$ again we get following lemma.
\begin{lemat}
For each $k\in\mathbb{N}$ it holds
\begin{equation}\begin{split}
\lim\limits_{l\to\infty}\int_0^T\int_\Omega&\ten{G}(\tilde{\theta} + \theta_{k,l},\tilde{\ten{T}}^d + \ten{T}_{k,l}^d):(\tilde{\ten{T}}^d + \ten{T}_{k,l}^d )\dxdt
= \int_0^T\int_\Omega\ten{G}(\tilde{\theta} + \theta_k,\tilde{\ten{T}}^d + \ten{T}_k^d):(\tilde{\ten{T}}^d + \ten{T}_k^d)\dxdt.
\end{split}\end{equation}
\end{lemat}
Now we are able to pass to the limit in the heat equation. Namely, we obtain the following equality for all test functions $\phi\in C^\infty([0,T]\times\Omega)$ 
\begin{equation}
\begin{split}
-\int_0^T\int_{\Omega} &\theta_k\phi_t \dxdt -
\int_{\Omega} \theta_k(x,0)\phi(x,0) \dx    + 
\int_0^T\int_{\Omega}  \nabla\theta_k\cdot\nabla\phi  \dxdt  
\\
&+ \int_Q \alpha\mbox{div}(\vc{u}_{k})_t \phi\dxdt = 
\int_0^T\int_{\Omega} \mathcal{T}_k\left((\ten{T}^d_{k}+\tilde{\ten{T}}^d):\ten{G}(\tilde{\theta} + \theta_{k},\tilde{\ten{T}}^d + \ten{T}^d_{k})\right)\phi \dxdt,
\\
\label{eq:after_limit_l_2}
\end{split}
\end{equation}
which completes the first limit passage.

\subsection{Boundedness of approximate solutions and limit passage with $k\to\infty$.}
\label{sec:2.3}

Since in previous section some uniform bounds were proved only with respect to $l$, we present here estimates with respect to $k$. At the beginning we recall two lemmas from \cite{1}, which are presented without proofs. Then we prove the third one which give us required estimates for time derivative of displacement.

\begin{lemat}{\cite[Lemma 3.6]{1}}

The sequence $\{\ten{\varepsilon}^{\bf p}_{k}\}$ is uniformly bounded in $W^{1,p'}(0,T,L^{p'}(\Omega,\mathcal{S}^3))$ with respect to $k$. 
\label{wsp_org_ep}
\end{lemat}

\begin{lemat}{\cite[Lemma 3.7]{1}}

The sequence $\{\vc{u}_{k}\}$ is uniformly bounded in $L^{p'}(0,T,W^{1,p'}_0(\Omega,\mathbb{R}^3))$ with respect to $k$.
\label{wsp_org_u}
\end{lemat}

\begin{lemat}
The sequence $\{(\vc{u}_{k})_t\}$ is uniformly bounded in $L^{p'}(0,T,W^{1,p'}_0(\Omega,\mathbb{R}^3))$.
\label{lm:2.10}
\end{lemat}

\begin{proof}
Multiplying \eqref{eq:11aa} by function $\varphi(t) \in C^{\infty}([0,T])$, integrating over time interval $(0,T)$ and passing to the limit with $l\to\infty$ we obtain
\begin{equation}
\int_0^T\int_{\Omega} \ten{D}(\ten{\varepsilon}(\vc{u}_{k}))_t: \ten{\varepsilon}(\vc{w}_{n}) \varphi(t) \dx\dt =
\int_0^T\int_{\Omega} \ten{D}(\ten{\varepsilon}^{\bf p}_{k})_t:
\ten{\varepsilon}(\vc{w}_{n}) \varphi(t) \dx \dt.
\label{eq:11aaa}
\end{equation}
Let us define the projection 
\begin{equation}
P^k: L^p(\Omega,\mathcal{S}^3) \to \mbox{lin}\{\ten{\varepsilon}(\vc{w}_1),...,\ten{\varepsilon}(\vc{w}_k)\}, \quad P^k\ten{v}:=\sum_{i=1}^{k}(\ten{v},\ten{\varepsilon}(\vc{w}_i))_{\ten{D}}\ten{\varepsilon}(\vc{w}_i).
\end{equation}
Let us take $\ten{\varphi} \in L^p(0,T,L^p(\Omega,\mathcal{S}^3))$. Property of projection implies that $\|P^k\ten{\varphi} \|_{L^{p}(0,T,L^{p}(\Omega,\mathcal{S}^3))}\leq \|\ten{\varphi} \|_{L^{p}(0,T,L^{p}(\Omega,\mathcal{S}^3))}$. Using the fact that $P^k(\ten{\varepsilon}(\vc{u}_{k}))_t=(\ten{\varepsilon}(\vc{u}_{k}))_t$ and using \eqref{eq:11aaa} we obtain that
\begin{equation}
\begin{split}
\int_0^T\int_{\Omega} \ten{D}(\ten{\varepsilon}(\vc{u}_{k}))_t: \ten{\varphi} \dx\dt &=
\int_0^T\int_{\Omega} \ten{D}(\ten{\varepsilon}(\vc{u}_{k}))_t: P^k\ten{\varphi} \dx\dt
= \int_0^T\int_{\Omega} \ten{D}(\ten{\varepsilon}^{\bf p}_{k})_t: P^k\ten{\varphi} \dx \dt
\\
& \leq 
\|(\ten{\varepsilon}^{\bf p}_{k})_t\|_{L^{p'}(0,T,L^{p'}(\Omega,\mathcal{S}^3))}
\|P^k\ten{\varphi} \|_{L^{p}(0,T,L^{p}(\Omega,\mathcal{S}^3))}
\\
& \leq 
\|(\ten{\varepsilon}^{\bf p}_{k})_t\|_{L^{p'}(0,T,L^{p'}(\Omega,\mathcal{S}^3))}
\|\ten{\varphi} \|_{L^{p}(0,T,L^{p}(\Omega,\mathcal{S}^3))}.
\end{split}
\end{equation}
Thus
\begin{equation}
\begin{split}
\|(\ten{\varepsilon}(\vc{u}_{k}))_t \|_{L^{p'}(0,T,L^{p'}(\Omega,\mathcal{S}^3))} & = \sup_{\ten{\varphi}\in L^p(0,T,L^p(\Omega,\mathcal{S}^3))\atop
\|\ten{\varphi}\|_{ L^p(0,T,L^p(\Omega,\mathcal{S}^3))}\leq 1}
\int_0^T\int_{\Omega} \ten{D}(\ten{\varepsilon}(\vc{u}_{k}))_t: \ten{\varphi} \dx\dt
\\
& \leq 
\|(\ten{\varepsilon}^{\bf p}_{k})_t\|_{L^{p'}(0,T,L^{p'}(\Omega,\mathcal{S}^3))}.
\end{split}
\end{equation}
This, together with Poincar\'{e}'s inequality and Lemma \ref{wsp_org_ep}, completes the proof.
\end{proof}

Let us focus on heat equation \eqref{eq:after_limit_l_2}. It is a weak formulation of equation 
\begin{equation}
(\theta_{k})_t - \Delta\theta_k 
 = 
\mathcal{T}_k\left((\tilde{\ten{T}}^d+\ten{T}^d_{k}):\ten{G}(\tilde{\theta} + \theta_{k},\tilde{\ten{T}}^d + \ten{T}^d_{k})\right) - \alpha{\div}(\vc{u}_{k})_t.
\label{eq:254}
\end{equation}
From previous estimates we get that right-hand side of above mentioned equation is uniformly bounded in $L^1(0,T,L^1(\Omega))$. Using Boccardo and Galllou\"{e}t approach for parabolic equation with Neumann boundary conditions, see Appendix \ref{sec:BG} or \cite{PhDFilip,1}, we observe that for each $1<q<\frac{5}{4}$ there exists $\theta \in L^q(0,T,W^{1,q}(\Omega))$ such that 
\begin{equation}
\theta_k \rightharpoonup \theta  \mbox{ weakly in } L^q(0,T,W^{1,q}(\Omega)). 
\end{equation}
Note that we need a weak convergence of right-hand side of \eqref{eq:254} to obtain strong convergence of $\{\theta_k\}$. Moreover, the rest of uniform estimates allows us to conclude that, at least for a subsequence, the following holds
\begin{equation}
\begin{array}{cl}
\theta_k \rightarrow \theta & \mbox{ a.e. in }\Omega\times(0,T),\\
\vc{u}_k \rightharpoonup  \vc{u} & \mbox{ weakly in } L^{p'}(0,T,W^{1,p'}_0(\Omega,\mathbb{R}^3)),\\
\ten{T}_k \rightharpoonup \ten{T} & \mbox{ weakly in }  L^2(0,T,L^2(\Omega,\mathcal{S}^3)),\\
\ten{T}^d_k \rightharpoonup \ten{T}^d & \mbox{ weakly in }  L^p(0,T,L^p(\Omega,\mathcal{S}^3_d)),\\
\ten{G}(\tilde{\theta} + \theta_k,\tilde{\ten{T}}^d +\ten{T}_k^d) \rightharpoonup \ten{\chi} & \mbox{ weakly in }  L^{p'}(0,T,L^{p'}(\Omega,\mathcal{S}^3_d)), \\
(\ten{\varepsilon}^{\bf p}_k)_t \rightharpoonup (\ten{\varepsilon}^{\bf p})_t & \mbox{ weakly in }  L^{p'}(0,T,L^{p'}(\Omega,\mathcal{S}^3_d)),
\\
(\ten{u}_{k})_t \rightharpoonup \ten{u}_t & \mbox{weakly in } L^{p'}(0,T,W^{1,p'}_0(\Omega,\mathbb{R}^3)).
\end{array}
\end{equation}
Consequently, passing to the limit in \eqref{limit1}, \eqref{65} we obtain
\begin{equation}\label{limit1a}
\begin{split}
\int_0^T\int_{\Omega}(\ten{T}-\alpha\ten{I}):\nabla\vc{\varphi} \dxdt =0
\end{split}
\end{equation}
\begin{equation}\label{65a}
\int_0^T\int_{\Omega}(\ten{\varepsilon}^{\bf p})_t : \ten{\psi}\dxdt = 
\int_0^T\int_{\Omega}\ten{\chi}  :  \ten{\psi}\dxdt
 \end{equation}
for all $ \ten{\varphi}\in C^\infty([0,T],W^{1,2}(\Omega,\mathcal{S}^3))$ (and then also for all $ \ten{\varphi}\in L^2(0,T;W^{1,2}(\Omega,\mathcal{S}^3))$) and for all $\ten{\psi}\in L^p(0,T;L^p(\Omega,\mathcal{S}^3))$. To complete this limit passage it remains to characterize the weak limit $\ten{\chi}$ and pass to the limit in the heat equation \eqref{eq:after_limit_l_2}. Again there is a problem with right-hand side of \eqref{eq:after_limit_l_2}. To deal with this issue we follow the similar lines as in the limit passage with $l\to\infty$. 

\begin{lemat}
The following inequality holds for the solution of approximate systems.
\begin{equation}
\limsup_{k\rightarrow\infty}\int_{0}^{t_2}\int_{\Omega}\ten{G}(\tilde{\theta} + \theta_{k},\tilde{\ten{T}}^d + \ten{T}^d_{k}):\ten{T}^d_k \dxdt \leq
\int_{0}^{t_2}\int_{\Omega}\ten{\chi}:\ten{T}^d \dxdt .
\label{jedna_nierownosc_1}
\end{equation}
\end{lemat}

\begin{proof}
Due to \eqref{82} we can repeat argumentation from the beginning of Lemma \ref{lm:8} proof and we obtain
\begin{equation}\label{do-l}
\frac{d}{dt} \mathcal{E}(\ten{\varepsilon}(\vc{u}_{k}) , \ten{\varepsilon}^{\bf p}_{k}) 
 = 
-
\int_{\Omega}\ten{G}(\tilde{\theta} + \theta_{k},\tilde{\ten{T}}^d + \ten{T}^d_{k}):\ten{T}^d_{k}\dx.
\end{equation}
We multiply the above identity by $\psi_{\mu,t_2}$ given by formula \eqref{psi-mu} and integrate over $(0,T)$. Passing to  the limit $k\to\infty$ we proceed in the same manner as in the proof of Lemma~\ref{lm:8} and obtain
\begin{equation}\label{mu4a}
\begin{split}
\liminf\limits_{k\to\infty}\int_{0}^{T}
\frac{d}{d\tau}& \mathcal{E}(\ten{\varepsilon}(\vc{u}_{k}) , \ten{\varepsilon}^{\bf p}_{k}) \,\psi_{\mu,t_2}\dt\\
&=\liminf\limits_{k\to\infty}\frac{1}{\mu}\int_{t_2}^{t_2+\mu}\mathcal{E}(\ten{\varepsilon}(\vc{u}_{k}) , \ten{\varepsilon}^{\bf p}_{k}) \dt-
\lim\limits_{k\to\infty}\mathcal{E}(\ten{\varepsilon}(\vc{u}_{k}(0)) , \ten{\varepsilon}^{\bf p}_{k}(0))\\
&\ge \frac{1}{\mu}\int_{t_2}^{t_2+\mu}\mathcal{E}(\ten{\varepsilon}(\vc{u}_{k}(t)) , \ten{\varepsilon}^{\bf p}_{k}(t)) \dt-
\mathcal{E}(\ten{\varepsilon}(\vc{u}(0)) , \ten{\varepsilon}^{\bf p}(0)).
\end{split}
\end{equation}

For the final step of this proof we need to show that the energy equality holds. Contrary to the case of previous section, we cannot use the time derivative of the limit, namely $(\vc{u})_t$ as the test function since is not regular enough. Although we can use a time derivative of approximate solution, i.e. $(\vc{u}_k)_t$. Moreover, due to the fact that $(\ten{\varepsilon}(\vc{u}_k) - \ten{\varepsilon}^{\bf p}_k)_t$ has worse regularity than $\ten{\varepsilon}(\vc{u}_k) - \ten{\varepsilon}^{\bf p}_k$ we will mollify it with respect to time. For $0<\epsilon < \min(t_1, T-t_2)$ let us take $\vc{\varphi}=((\ten{\varepsilon}(\vc{u}_k)*\eta_\epsilon)_t\mathbf{1}_{(t_1,t_2)})*\eta_{\epsilon}$ as test function in \eqref{limit1a}. Here, $\eta_\epsilon$ is a standard mollifier and we mollify with respect to time. Then
\begin{equation}
\int_{t_1}^{t_2} \int_{\Omega}(\ten{D}(\ten{\varepsilon}(\vc{u}) - \ten{\varepsilon}^{\bf p})-\alpha\ten{I})*\eta_{\epsilon}: (\ten{\varepsilon}(\vc{u}_k)*\eta_{\epsilon})_t \dxdt =
0.
\label{pierwsze_r2}
\end{equation}
As previously, term with $\alpha\ten{I}$ is equal to zero. Then we test an approximate equation \eqref{65} by a test function $\ten{\psi}=(\ten{T}_k^d*\eta_{\epsilon}\mathbf{1}_{(t_1,t_2)})*\eta_{\epsilon}$. This provides to 
\begin{equation}
\int_{t_1}^{t_2}\int_{\Omega}(\ten{\varepsilon}^{\bf p}_{k}*\eta_{\epsilon})_t:\ten{T}*\eta_{\epsilon} \dxdt =
\int_{t_1}^{t_2}\int_{\Omega}\ten{G}(\tilde{\theta}+\theta_k,\tilde{\ten{T}}^d+\ten{T}^d_k)*\eta_{\epsilon}:\ten{T}*\eta_{\epsilon} \dxdt .
\label{drugie_r2a}
\end{equation}
Subtracting \eqref{drugie_r2a} from \eqref{pierwsze_r2} we get
\begin{equation}
\int_{t_1}^{t_2}\int_{\Omega}\ten{T}*\eta_{\epsilon}:(\ten{\varepsilon}(\vc{u}_k) - \ten{\varepsilon}^{\bf p}_k)_t*\eta_{\epsilon} \dxdt=
- 
\int_{t_1}^{t_2}\int_{\Omega}\ten{G}(\tilde{\theta}+\theta_k,\tilde{\ten{T}}^d + \ten{T}^d_k)*\eta_{\epsilon}:\ten{T}^d*\eta_{\epsilon} \dxdt .
\label{granica_k_ptrzed}
\end{equation}
For every $\epsilon>0$ the sequence $\{(\ten{\varepsilon}(\vc{u}_k) - \ten{\varepsilon}^{\bf p}_k)_t*\eta_{\epsilon}\}$ belongs to $L^2(0,T,L^2(\Omega,\mathcal{S}^3))$ and is uniformly bounded in $L^2(0,T,L^2(\Omega,\mathcal{S}^3))$, hence we pass to the limit with $k\rightarrow\infty$. Using the properties of convolution, passing to the limit with $\epsilon \to 0$ and then with $t_1 \to 0$ we obtain
\begin{equation}
\int_{\Omega}\ten{D}(\ten{\varepsilon}(\vc{u}) - \ten{\varepsilon}^{\bf p}):(\ten{\varepsilon}(\vc{u}) - \ten{\varepsilon}^{\bf p}) \dx \Big|_{0}^{t_2}=
- 
\int_{0}^{t_2}\int_{\Omega}\ten{\chi}:\ten{T}^d \dxdt .
\label{granica_k}
\end{equation}
We multiply \eqref{granica_k} by $\frac{1}{\mu}$ and integrate over $(t_2,t_2+\mu)$ and proceed now in the same manner as in the proof of Lemma~\ref{lm:8}.
\end{proof}

Using the Minty-Browder trick to identify the weak limit $\ten{\chi}$ and the same argumentation as in the previous limit passage, we obtain that $\ten{\chi} =\ten{G}(\tilde{\theta} + \theta,\tilde{\ten{T}}^d +\ten{T}^d)$ a.a. in $\Omega\times (0,T)$ and 
\begin{equation}
\ten{G}(\tilde{\theta} + \theta_k,\tilde{\ten{T}}^d + \ten{T}_k^d):(\tilde{\ten{T}}^d + \ten{T}^d_{k})\rightharpoonup \ten{G}(\tilde{\theta} + \theta,\tilde{\ten{T}^d} +\ten{T}^d):(\tilde{\ten{T}}^d + \ten{T}^d) 
\quad\mbox{ in }L^1(0,T,L^1(\Omega)).
\end{equation}
Furthermore,
\begin{equation}
\mathcal{T}_k\Big(\ten{G}(\tilde{\theta} + \theta_k,\tilde{\ten{T}}^d + \ten{T}_k^d):(\tilde{\ten{T}}^d + \ten{T}^d_{k})\Big)\rightharpoonup \ten{G}(\tilde{\theta} + \theta,\tilde{\ten{T}}^d +\ten{T}^d):(\tilde{\ten{T}}^d + \ten{T}^d),
\end{equation}
in $L^1(0,T,L^1(\Omega))$. Then passing to the limit with $k\to \infty$ in \eqref{eq:after_limit_l_2} we obtain
\begin{equation}
\begin{split}
-\int_0^T\int_{\Omega} &\theta\phi_t \dxdt -
\int_{\Omega} \theta(x,0)\phi(x,0) \dx    + 
\int_0^T\int_{\Omega}  \nabla\theta\cdot\nabla\phi  \dxdt  
\\
&+ \int_Q \alpha\mbox{div}(\vc{u})_t \phi\dxdt = 
\int_0^T\int_{\Omega} (\tilde{\ten{T}}^d + \ten{T}^d):\ten{G}(\tilde{\theta} + \theta,\tilde{\ten{T}} + \ten{T}^d)\phi \dxdt,
\\
\label{eq:after_limit_l_2a}
\end{split}
\end{equation}
for all $\phi\in C^\infty([0,T]\times\Omega)$ which completes the second limit passage and as a consequence it finishes the proof of Theorem \ref{thm:main}.

%
%
%
%

\begin{appendix}
\section{Transformation into homogeneous boundary value problem}
\label{sec:trans}

The aim of this section is to reduce full boundary problem into homogeneous one. Let us define two additional systems of equations
\begin{equation}
\left\{
\begin{array}{rcll}
-\div \tilde{\ten{T}} &=& \vc{f} & \mbox{in } \Omega\times (0,T), \\
\tilde{\ten{T}} &=& \ten{D}\ten{\varepsilon}(\tilde{\vc{u}}) & \mbox{in } \Omega\times (0,T), \\
\tilde{\vc{u}} &=& \vc{g} & \mbox{on } \partial\Omega\times (0,T), 
\end{array}
\right.
\label{war_brz_u}
\end{equation}
and
\begin{equation}
\left\{
\begin{array}{rcll}
\tilde{\theta}_t -\Delta \tilde{\theta} +\alpha \div \tilde{\vc{u}}_t&=& 0 & \mbox{in } \Omega\times (0,T), \\
\frac{\partial\tilde{\theta}}{\partial\vc{n}} &=& g_{\theta} & \mbox{on } \partial\Omega\times (0,T), \\
\tilde{\theta}(x,0) &=& \tilde{\theta}_0 & \mbox{in } \Omega,
\end{array}
\right.
\label{war_brz_t}
\end{equation}
where $\vc{f}$ is a given volume force, $\vc{g}$ and $g_\theta$ are given boundary values for displacement and thermal flux, respectively. It may be understand as follows: system \eqref{war_brz_u} is subject to the same external forces as problem \eqref{full_system}. Since \eqref{war_brz_u} describes elastic deformation, no mechanical energy is transformed into thermal one. That is the reason why right-hand side of \eqref{war_brz_t}$_{(1)}$ is equal to zero. \eqref{war_brz_u}--\eqref{war_brz_t} are complemented with the same boundary conditions as \eqref{full_system} and $\tilde{\theta}_0\in L^2(\Omega)$ is arbitrary function. Moreover, one should remember that inelastic deformation in \eqref{full_system} is defined by Norton-Hoff-type constitutive function (i.e. it satisfies $p$-growth condition with respect to second variable) with $p\geq 2$.

\begin{lemat}
Let $\tilde{\theta}_0 \in L^2(\Omega)$, $\vc{g} \in W^{1,p}(0,T, W^{1-\frac{1}{p},p}(\partial\Omega,\mathbb{R}^3))$, $g_{\theta} \in L^2(0,T,L^2(\partial\Omega))$ and moreover  $\vc{f}\in W^{1,p}(0,T,W^{-1,p}(\Omega,\mathbb{R}^3))$. Then there exists a solution to systems \eqref{war_brz_u} and \eqref{war_brz_t}. Additionally, the following estimates hold:
\begin{equation}
\begin{split}
\|\tilde{\vc{u}}\|_{W^{1,p}(0,T,W^{1,p}(\Omega))} 
& \leq 
C_1 \left(\|\vc{g}\|_{W^{1,p}(0,T, W^{1-\frac{1}{p},p}(\partial\Omega)}+ 
\|\vc{f}\|_{W^{1,p}(0,T,W^{-1,p}(\Omega))} \right),
 \\
\|\tilde{\theta}\|_{L^{\infty}(0,T,L^2(\Omega))}  + \|\tilde{\theta}\|_{L^2(0,T,W^{1,2}(\Omega))} 
& \leq 
C_2 \left(\|g_{\theta}\|_{L^2(0,T,L^2(\partial\Omega))}+\|\tilde{\theta}_0\|_{L^2(\Omega)} +\alpha\|\div \tilde{\vc{u}}_t\|_{L^p(0,T,L^p(\Omega))} \right). 
\nonumber
\end{split}
\end{equation}
\label{wyrzucenie_war_brzeg}
Moreover, $\theta$ belongs to $C([0,T],L^2(\Omega))$.
\end{lemat}

\begin{uwaga}
From the trace theorem \cite[Chapter II]{Valent} there exist $\tilde{\vc{g}}\in W^{1,p}(0,T,W^{1,p}(\Omega,\mathbb{R}^3))$ such that $\tilde{\vc{g}}|_{\partial\Omega}=\vc{g}$. Then, finding the solution $\tilde {\vc{u}}$ to \eqref{war_brz_u} is equivalent to finding the solution  $\tilde{\vc{u}}_1$ to the following problem 
\begin{equation}
\left\{
\begin{array}{rcll}
-{\div} \ten{D}\ten{\varepsilon}(\tilde{\vc{u}}_1) &=& \vc{f} +{\div} \ten{D}\ten{\varepsilon}(\vc{\tilde{g}}) & \mbox{in } \Omega\times (0,T), \\
\tilde{\vc{u}}_1 &=& 0 & \mbox{on } \partial\Omega\times (0,T),
\end{array}
\right.
\label{war_brz_u_0}
\end{equation}
and $\tilde{\vc{u}} = \tilde{\vc{u}}_1 + \tilde{\vc{g}}$. Using \cite[Corollary 4.4]{Valent} for \eqref{war_brz_u_0} and for time derivative of this equation, we obtain the estimates on $\tilde{\vc{u}}$ presented in Lemma \ref{wyrzucenie_war_brzeg}, whereas the estimates for $\tilde{\theta}$ are standard calculations (note that $p\geq 2$).
\end{uwaga}

Then, instead of finding  $(\widehat{\vc u}, \widehat{\theta})$ - the solution to problem \eqref{full_system}-\eqref{init_0}-\eqref{boun_0}, we shall search for $(\vc{u}, \theta)$, where $\vc{u}=\widehat{\vc{u}}-\tilde{\vc{u}}$, $\theta=\widehat{\theta}-\tilde{\theta}$ and $(\tilde{\vc{u}},\tilde{\theta})$ solve \eqref{war_brz_u} and \eqref{war_brz_t}. It means that we consider \begin{equation}
\left\{
\begin{split}
- {\div} \ten{\sigma} & =  0 ,
\\
\ten{T} & = \ten{D}(\ten{\varepsilon}(\vc{u}) - \ten{\varepsilon}^{\bf p} ),
\\
\ten{\sigma} & = \ten{T}-\alpha\ten{I},
\\
\ten{\varepsilon}^{\bf p}_t & =  \ten{G}(\tilde{\theta} + \theta ,  \tilde{\ten{T}}^d + \ten{T}^d),
\\
\theta_t - \Delta \theta + \alpha \div (\vc{u}_t)& =   \big(\tilde{\ten{T}}^d + \ten{T}^d\big):\ten{G}(\tilde{\theta} + \theta , \tilde{\ten{T}}^d + \ten{T}^d),
\end{split}
\right.
\label{full_system_22}
\end{equation}
with initial and boundary conditions
\begin{equation}
\left\{	
\begin{array}{rcll}
\theta(\cdot,0) &=& \widehat{\theta}_0 - \tilde{\theta}_0 \equiv \theta_0 & \mbox{in } \Omega, \\
\ten{\varepsilon}^{\bf p}(\cdot,0) &=& \ten{\varepsilon}^{\bf p}_0 & \mbox{in } \Omega,
\\
\vc{u} &=& 0 & \mbox{on } \partial\Omega\times (0,T), \\
\frac{\partial\theta}{\partial \vc{n}} &=& 0 & \mbox{on } \partial\Omega\times (0,T),
\end{array}
\right.
\label{in_bou_condA5}
\end{equation}
where $\widehat{\theta}_0$ is given initial condition for the temperature and $\tilde{\theta}_0$ is initial condition for the system \eqref{war_brz_t}.

%
%
%
%

\section{Construction of approximate solution}
\label{sec:constr}

Construction of approximate solutions is done in the same way as in \cite{1}. There are no issues with bases for temperature and displacement. Special attention is required in the construction of basis for visco-elastic strain tensor $\ten{\varepsilon}^{\bf p}$. For more details we refer the reader to \cite[Appendix B]{1}. Here, we briefly summarized the results presented there. Let $k\in{\mathbb N}$ and $\mathcal{T}_k(\cdot)$ be a standard truncation operator
\begin{equation}
\mathcal{T}_k(x)=\left\{
\begin{split}
k \qquad & x> k \\
x \qquad & |x|\leq k \\
-k \qquad & x <-k.
\end{split}
\right.
\label{Tk}
\end{equation}

Since the right-hand side and initial condition of heat equation are only the integrable function, we use two level approximation, i.e. independent parameters of approximation in the displacement and temperature. Further, approximate solution will be denoted by index $(k,l)$, where $k$ corresponds to range of Galerkin approximation of displacement and $l$ corresponds to range of Galerkin approximation of temperature. This allows us to make limit passages independent for both approximations.

Now, we construct basis for approximate solutions for displacement. Let us consider the space $L^2(\Omega,\mathcal{S}^3)$ with a scalar product defined
\begin{equation}
(\ten{\xi},\ten{\eta})_{\ten{D}}:=  \int_\Omega {\ten{D}}^\frac{1}{2}\ten{\xi}: {\ten{D}}^\frac{1}{2}\ten{\eta} \dx 
\quad\mbox{for }\ten{\xi},\ten{\eta}\in L^2(\Omega,\mathcal{S}^3)
 \end{equation}
 where ${\ten{D}}^\frac{1}{2}\circ{\ten{D}}^\frac{1}{2}=\ten{D}$. 
Let $\{\vc{w}_i\}_{i=1}^{\infty}$ be the set of eigenfunctions of the operator $-\div\ten{D}\ten{\varepsilon}(\cdot)$ with the domain $W_0^{1,2}(\Omega,\mathbb{R}^3)$ and  $\{ \lambda_i \}$ be the corresponding eigenvalues such that $\{\vc{w}_i\}$ is orthonormal in $W^{1,2}_0(\Omega,\mathbb{R}^3)$ with the inner product
\begin{equation}
( \vc{w}, \vc{v})_{W^{1,2}_0(\Omega)}=( \ten{\varepsilon}(\vc{w}), \ten{\varepsilon}(\vc{v}))_{\ten{D}}
\end{equation}
and orthogonal in $L^2(\Omega,\mathbb{R}^3)$. Since we assume that each of the function $d_{i,j,k,l}$ is constant and boundary of the domain is $C^2$, we know that the basis $\{\ten{w}_i\}$ consists of functions which belong to $H^{3}(\Omega, \mathbb{R}^3)$, see \cite{Brezis}. Let us denote 
\begin{equation}
\|\ten{\varepsilon}(\vc{w})\|_{\ten{D}}:=\sqrt{( \ten{\varepsilon}(\vc{w}), \ten{\varepsilon}(\vc{w}))_{\ten{D}}}.
\end{equation}
Using the eigenvalue problem for the operator $-\div\ten{D}\ten{\varepsilon}(\cdot)$ we obtain
\begin{equation}
\int_{\Omega}\ten{D}\ten{\varepsilon}(\vc{w}_i):\ten{\varepsilon}(\vc{w}_j) \dx = \lambda_i \int_{\Omega}\vc{w}_i\cdot\vc{w}_j \dx = 0.
\end{equation} 
Furthermore, let $\{v_i\}_{i=1}^\infty$ be the subset of $W^{1,2}(\Omega)$ such that
\begin{equation}
\int_{\Omega} (\nabla v_i \cdot\nabla\phi - \mu_i v_i\phi)\dx =0,
\end{equation}
holds for every function $\phi\in C^{\infty}(\overline{\Omega})$, see \cite{Alt,strauss}. 
We may assume that $\{v_i\}$ is orthonormal in $W^{1,2}(\Omega)$ and orthogonal in $L^2(\Omega)$. Let $\{\mu _i \}$ be the set of corresponding eigenvalues. The set $\{\vc{w}_i\}$ is used to construct approximate solutions of displacement, whereas set $\{v_i\}$ is used to construct approximate solutions of temperature. What remains, is to construct basis for visco-elastic strain tensor.

Let us consider the symmetric gradients of first $k$ functions from the basis $\{\ten{w}_i\}_{i=1}^{\infty}$. Due to the regularity of eigenfunctions we observe that $\ten{\varepsilon}(\ten{w}_i)$ are elements of $H^s(\Omega,\mathcal{S}^3)$, that is fractional Sobolev space with a scalar product denoted by $\braket{\cdot,\cdot}_s$ for $\frac{3}{2}<s\le 2$. We define space 
\begin{equation}\label{Vk}
V_k:= (\mbox{span}\{\ten{\varepsilon}(\ten{w}_1),...,\ten{\varepsilon}(\ten{w}_k)\})^\bot,
\end{equation}
which is the orthogonal complement in $L^2(\Omega,\mathcal{S}^3)$  taken with respect to the scalar product $(\cdot,\cdot)_{\ten{D}}$. Then let us introduce the space
\begin{equation}\label{Vks}
V_k^s:=V_k\cap H^s(\Omega,\mathcal{S}^3).
\end{equation}
Since the co-dimension of $V_k^s$ is finite, then $V_k^s$ is closed in $H^s(\Omega,\mathcal{S}^3)$ with respect to the  $\|\cdot\|_{H^s}-$norm. Using \cite[Theorem B.1]{1} we construct the orthogonal basis of $V_k$,
which is also an orthonormal basis of $V_k^s$. We denote this basis by $\{\ten{\zeta}^k_n\}_{n=1}^{\infty}$. 

For $k,l\in\mathbb{N}$, we are ready to define approximate solution
\begin{equation}
\begin{split}
\vc{u}_{k,l} & = \sum_{n=1}^k\alpha_{k,l}^n(t) \vc{w}_n,
 \\
\theta_{k,l} & = \sum_{m=1}^l\beta_{k,l}^m(t) v_m,
 \\
\ten{\varepsilon}^{\bf p}_{k,l} & = \sum_{n=1}^k\gamma_{k,l}^n(t) \ten{\varepsilon}(\vc{w}_n) + 
\sum_{m=1}^l\delta_{k,l}^m(t) \ten{\zeta}_m^k.
\end{split}
\label{eq:postac}
\end{equation}

%
%
%
%

\section{Proof of Lemma \ref{wsp_org_epa}}
\label{sec:Lemma}
The following proof comes from \cite{1}.
\begin{proof}{}

Recall that $\{\ten{\zeta}_n^k\}_{n=1}^\infty$ is an orthonormal basis of $V_k^s$ and orthogonal basis of $V_s$, where those spaces are defined by \eqref{Vks} and \eqref{Vk}, respectively. We define the following projections:
\begin{equation}
\begin{split}
P^l_{H^s}: H^s &\to \mbox{lin}\{\ten{\zeta}^k_1,...,\ten{\zeta}^k_l\}, \quad P^l_{H^s} \ten{v}:=\sum_{i=1}^{l} \braket{ \ten{v},\frac{\ten{\zeta}_i^k}{\sqrt{\lambda_i}}}_s\frac{\ten{\zeta}_i^k}{\sqrt{\lambda_i}} \\
P^l_{L^2}: L^2 &\to \mbox{lin}\{\ten{\zeta}^k_1,...,\ten{\zeta}^k_l\}, \quad P^l_{L^2}\ten{v}:=\sum_{i=1}^{l}(\ten{v},\ten{\zeta}_i^k)_{\ten{D}}\ten{\zeta}_i^k.
\end{split}
\end{equation}
Then, we observe that
\begin{equation}
P^l_{L^2}{\Big|_{V^s_k}} = P^l_{H^s}{\Big|_{V^s_k}}.
\end{equation}
Indeed, if $\ten{\varphi}\in V_k^s$ then
\begin{equation}
P^l_{L^2}\ten{\varphi}=\sum_{i=1}^{l}(\ten{\varphi},\ten{\zeta}_i^k)_{\ten{D}}\ten{\zeta}_i^k =
\sum_{i=1}^{l}\braket{\ten{\varphi},\frac{\ten{\zeta}_i^k}{\sqrt{\lambda_i}}}_s\frac{\ten{\zeta}_i^k}{\sqrt{\lambda_i}}=
P^l_{H^s}\ten{\varphi},
\end{equation}
where the second equality is condition for eigenvalues.  Moreover, the norms $\|P^l_{H^s}\|_{\mathcal{L}(H^s)}$ and $\|P^l_{L^2}\|_{\mathcal{L}(L^2)}$ are equal to $1$. 

Let us define the projection
\begin{equation}
P^k: L^2 \to \mbox{lin}\{\ten{\varepsilon}(\vc{w}_1),...,\ten{\varepsilon}(\vc{w}_k)\}, \quad P^k\ten{v}:=\sum_{i=1}^{k}(\ten{v},\ten{\varepsilon}(\vc{w}_i))_{\ten{D}}\ten{\varepsilon}(\vc{w}_i).
\end{equation}
Our goal is to obtain the estimates independent of $l$. Since $P^k$ is the projection which does not dependent on $l$, then there exists $c(k)$ (depending only on $k$) such that for every $\ten{\varphi}\in\mathcal{S}^3$ it holds
\begin{equation}
\max(\|P^k\ten{\varphi}\|_{H^s}, \|(Id - P^k)\ten{\varphi}\|_{H^s})\le c(k)\|\ten{\varphi}\|_{H^s}.
\end{equation}
Thus, we may observe that
\begin{equation}
P^l_{H^s}(Id - P^k)\ten{v} =
\sum_{i=1}^{l}\braket{(Id - P^k)\ten{v},\frac{\ten{\zeta}_i^k}{\sqrt{\lambda_i}}}_s\frac{\ten{\zeta}_i^k}{\sqrt{\lambda_i}}
=
\sum_{i=1}^{l}((Id - P^k)\ten{v},\ten{\zeta}_i)_{\ten{D}}\ten{\zeta}_i^k
=
\sum_{i=1}^{l}(\ten{v},\ten{\zeta}_i)_{\ten{D}}\ten{\zeta}_i^k
=
P^l_{L^2}\ten{v}.
\end{equation}

Notice that by \eqref{eq:postac}$_{(3)}$ we have $(P^k + P^l_{L^2})(\ten{\varepsilon}^{\bf p}_{k,l})_t=(\ten{\varepsilon}^{\bf p}_{k,l})_t$.
 Let $\ten{\varphi}\in L^p(0,T,H^{s}(\Omega,\mathcal{S}^3))$, then we may estimate as follows
\begin{equation}
\begin{split}
\int_0^T |( (\ten{\varepsilon}^{\bf p}_{k,l})_t, \ten{\varphi} )_{\ten{D}} |\dt &=
\int_0^T |( (P^k + P^l_{L^2})(\ten{\varepsilon}^{\bf p}_{k,l})_t, \ten{\varphi})_{\ten{D}} |\dt
\\
&=
\int_0^T |( (\ten{\varepsilon}^{\bf p}_{k,l})_t, (P^k + P^l_{L^2})\ten{\varphi})_{\ten{D}} |\dt
\\ &
\le\int_0^T |( (\ten{\varepsilon}^{\bf p}_{k,l})_t, P^k\ten{\varphi})_{\ten{D}} |\dt
+\int_0^T |( (\ten{\varepsilon}^{\bf p}_{k,l})_t, P^l_{L^2}\ten{\varphi})_{\ten{D}} |\dt ,
\end{split}
\end{equation}
where the second equality holds because the projections are self-adjoint operators and the inequality is a consequence of the orthogonality of subspaces $\mbox{lin}\{\ten{\varepsilon}(\vc{w}_1),\ldots,\ten{\varepsilon}(\vc{w}_k)\}$ and $\mbox{lin}\{\ten{\zeta}_1^k,\ldots, \ten{\zeta}_l^k\}$ in the sense of $(\cdot,\cdot)_{\ten{D}}$. Thus,
\begin{equation}
\begin{split}
\int_0^T |( (\ten{\varepsilon}^{\bf p}_{k,l})_t, \varphi )_{\ten{D}} |\dt &\le
\int_0^T |\int_\Omega
\ten{D}\ten{G}(\theta_{k,l}+\tilde{\theta},\ten{T}_{k,l}^d +\tilde{\ten{T}}^d) P^k\ten{\varphi}\dx |\dt
\\ &
\quad +
\int_0^T |\int_\Omega
\ten{D}\ten{G}(\theta_{k,l}+\tilde{\theta},\ten{T}_{k,l}^d +\tilde{\ten{T}}^d) P^l_{L^2}\ten{\varphi}\dx |\dt
\\ &\le 
\int_0^T |\int_\Omega
\ten{D}\ten{G}(\theta_{k,l}+\tilde{\theta},\ten{T}_{k,l}^d +\tilde{\ten{T}}^d) P^k\ten{\varphi}\dx |\dt
\\ &
\quad +
\int_0^T |\int_\Omega
\ten{D}\ten{G}(\theta_{k,l}+\tilde{\theta},\ten{T}_{k,l}^d +\tilde{\ten{T}}^d) (P^l_{H^s}\circ(Id - P^k))\ten{\varphi}\dx |\dt.
\end{split}
\label{eq:drugi_wyraz}
\end{equation}
The estimates of this first term on right hand side of abovementioned inequality are obvious
\begin{equation}
\begin{split}
\int_0^T|\int_{\Omega}\ten{D}\ten{G}(\theta_{k,l}+\tilde{\theta},\ten{T}_{k,l}^d +\tilde{\ten{T}}^d)
& P^k\ten{\varphi} |\dt 
\\
& \leq d\int_0^T\|\ten{G}(\theta_{k,l}+\tilde{\theta},\ten{T}_{k,l}^d +\tilde{\ten{T}}^d)\|_{L^{p'}(\Omega)}
\|P^k\ten{\varphi}\|_{L^{p}(\Omega)}\dt 
\\
& \le  \tilde c\int_0^T\|\ten{G}(\theta_{k,l}+\tilde{\theta},\ten{T}_{k,l}^d +\tilde{\ten{T}}^d)\|_{L^{p'}(\Omega)}
\|P^k\ten{\varphi}\|_{H^{s}(\Omega)}\dt
\\ &
\le c(k)\tilde c\int_0^T\|\ten{G}(\theta_{k,l}+\tilde{\theta},\ten{T}_{k,l}^d +\tilde{\ten{T}}^d)\|_{L^{p'}(\Omega)}
\|\ten{\varphi}\|_{H^{s}(\Omega)}\dt
\\
&\le c(k)\tilde c\|\ten{G}(\theta_{k,l}+\tilde{\theta},\ten{T}_{k,l}^d +\tilde{\ten{T}}^d)\|_{L^{p'}(0,T,L^{p'}(\Omega))}
\|\ten{\varphi}\|_{L^p(0,T,H^{s}(\Omega))},
\end{split}
\end{equation}
where  $\tilde c$ is an optimal embedding  constant of $H^s(\Omega,\mathcal{S}^3)\subset L^p(\Omega,\mathcal{S}^3)$. Now, let us focus on the second term from \eqref{eq:drugi_wyraz}. We obtain
\begin{equation}
\begin{split}
\int_0^T |\int_\Omega
\ten{D}\ten{G}(\theta_{k,l}+\tilde{\theta},\ten{T}_{k,l}^d +\tilde{\ten{T}}^d) &(P^l_{H^s}\circ(Id - P^k))\ten{\varphi}\dx |\dt  \\
& \leq 
d\int_0^T \|
\ten{G}(\theta_{k,l}+\tilde{\theta},\ten{T}_{k,l}^d +\tilde{\ten{T}}^d)\|_{L^{p'}(\Omega)}
\|(P^l_{H^s}\circ(Id - P^k))\ten{\varphi}\|_{L^{p}(\Omega)}\dt 
\\
& \leq \tilde c
\int_0^T \|
\ten{G}(\theta_{k,l}+\tilde{\theta},\ten{T}_{k,l}^d +\tilde{\ten{T}}^d)\|_{L^{p'}(\Omega)}
\|(P^l_{H^s}\circ(Id - P^k))\ten{\varphi}\|_{H^{s}(\Omega)}\dt 
\\
& \leq \tilde c
\int_0^T \|
\ten{G}(\theta_{k,l}+\tilde{\theta},\ten{T}_{k,l}^d +\tilde{\ten{T}}^d)\|_{L^{p'}(\Omega)}
\|(Id - P^k)\ten{\varphi}\|_{H^{s}(\Omega)}\dt 
\\
& \leq \tilde c c(k)
\int_0^T \|
\ten{G}(\theta_{k,l}+\tilde{\theta},\ten{T}_{k,l}^d +\tilde{\ten{T}}^d)\|_{L^{p'}(\Omega)}
\|\ten{\varphi}\|_{H^{s}(\Omega)}\dt 
\\
&
\le \tilde c c(k)\|\ten{G}(\theta_{k,l}+\tilde{\theta},\ten{T}_{k,l}^d +\tilde{\ten{T}}^d)\|_{L^{p'}(0,T,L^{p'}(\Omega))}
\|\ten{\varphi}\|_{L^p(0,T,H^{s}(\Omega))}.
\end{split}
\end{equation}
Consequently, there exists $C(k)>0$ such that 
\begin{equation}
\sup_{\ten{\varphi}\in L^p(0,T,H^{s}(\Omega))\atop
 \|\ten{\varphi}\|_{L^p(0,T,H^{s}(\Omega))}\le 1 }\int_0^T
|( (\ten{\varepsilon}^{\bf p}_{k,l})_t, \ten{\varphi})_{\ten{D}} |\dt\le C(k),
\end{equation}
and hence sequence $\{(\ten{\varepsilon}^{\bf p}_{k,l})_t\}$ is uniformly bounded in 
${L^{p'}(0,T,(H^{s}(\Omega,\mathcal{S}^3))')}$ with respect to~$l$.

\end{proof}

%
%
%
%

\section{Solution to heat equations}
\label{sec:BG}
Let us focus on the following problem 
\begin{equation}
\left\{
\begin{array}{ll}
(\theta)_t -\Delta \theta = f & \mbox{in } \Omega\times (0,T),\\
\frac{\partial \theta}{\partial \vc{n}} & \mbox{on } \partial\Omega\times (0,T),\\
\theta(\cdot,0) = \theta_0 & \mbox{on } \Omega,
\end{array}
\right.
\label{eq:heatB}
\end{equation}
where $f$ belongs to $L^1(\Omega\times (0,T))$ and $\theta_0$ belongs to $L^1(\Omega)$. Now, let us define the approximate system of equations in the following way: for every $k \in\mathbb{N}$ let us take function $f_k$ which belongs to $L^2(\Omega\times (0,T))$ and the sequence $\{f_k\}$ is uniformly bounded in $L^1(\Omega\times (0,T))$. Additionally, $\mathcal{T}_k (\theta_0 )\in L^2(\Omega)$, $\|\mathcal{T}_k(\theta_0)\|_{L^1(\Omega)}\leq \|\theta_0\|_{ L^1(\Omega)}$ and $\mathcal{T}_k (\theta_0 ) \to \theta_0$ in $L^1(\Omega)$. Then $\theta_k$ is an approximate solution of
\begin{equation}
\left\{
\begin{array}{ll}
(\theta_k)_t -\Delta \theta_k = f_k & \mbox{in } Q,\\
\frac{\partial \theta_k}{\partial \vc{n}} = 0 & \mbox{on } \partial\Omega\times (0,T),\\
\theta_k(\cdot,0) = \mathcal{T}_k(\theta_0) & \mbox{on } \Omega,
\end{array}
\right.
\label{eq:heatBapp}
\end{equation}
where $\mathcal{T}_k (\cdot)$ is a standard truncation operator defined in \eqref{Tk}.

\begin{lemat}
Let $\{f_k\}$ be uniformly bounded in $L^1(\Omega\times (0,T))$ and $\theta_0$ belongs to $L^1(\Omega)$. Then, for $q < \frac{2(N +1)-N}{N+1}$ ($q < \frac{5}{4}$ for $N = 3$) the sequence $\{\theta_k\}$ is uniformly bounded in $L^q (0, T, W^{1,q} (\Omega))$. Moreover, if $\{f_k\}$ converges weakly to $f$ in $L^1(\Omega\times (0,T))$ then $\theta$ belongs to $L^q (0, T, W^{1,q} (\Omega)) \cap C([0, T ], W^{-2,2} (\Omega))$ and it is solution to the system \eqref{eq:heatB}.
\label{lm:BG1}
\end{lemat}
The idea of Lemma \ref{lm:BG1} proof is simple. Firstly, we used truncation of approximate solutions as test function to prove uniform estimates in $L^q (0, T, W^{1,q} (\Omega))$. The vale of $q$ is a consequence of summable of series which appear during the estimates. The uniform boundedness of this sequence gives us pointwise convergence of temperature. To prove the second part of lemma we need to have at least weak convergence of right-hand side in $L^1(\Omega\times(0,T))$. For Norton-Hoff-type model with linear thermal expansion the right-hand side function is defined in the following way
\begin{equation}
f_k = \ten{G}(\tilde{\theta}+\theta_k, \tilde{\ten{T}}^d + \ten{T}^d_k):(\tilde{\ten{T}}^d + \ten{T}^d_k) - \alpha \mbox{div} (\vc{u}_k)_t.
\end{equation}
To obtain the convergence of this sequence we used a pointwise convergence of sequence $\{\theta_k\}$, see Section \ref{sec:2.3}. Finally, using compactness argument we get that $\theta  \in L^q (0, T, W^{1,q} (\Omega)) \cap C([0, T ], W^{-2,2} (\Omega))$. The original idea of such estimates (for Dirichlet boundary) problem was presented in \cite{Boccardo}. Proof of Lemma \ref{lm:BG1} (heat equation with Neumann boundary condition) may be found in \cite{PhDFilip}. 

\end{appendix}


\end{document}